\documentclass[final]{article}
\usepackage{amsmath,mathrsfs,amssymb,amsthm,amscd}
\usepackage{color,epsfig,latexsym,graphicx,eucal,layout,fancyhdr}
\usepackage{enumerate}
\usepackage[]{fontenc}
\usepackage[all]{xy}
\usepackage{hyperref}
\usepackage{graphics}
\usepackage{a4wide}
\usepackage{psfrag}
\usepackage{todonotes}
\newcounter{EQNR}
\renewcommand{\contentsname}{Contents\\
{\footnotesize\normalfont(A table of contents should normally not be included)}}

\newtheorem{theorem}{Theorem}

\newtheorem{corollary}[theorem]{Corollary}

\newtheorem{definition}[theorem]{Definition}

\newtheorem{lemma}[theorem]{Lemma}

\newtheorem{proposition}[theorem]{Proposition}
\newtheorem{remark}[theorem]{Remark}

\DeclareMathOperator{\PSL}{PSL}
\DeclareMathOperator{\hyp}{hyp}
\let\Im\relax
\DeclareMathOperator{\Im}{Im}
\let\Re\relax
\DeclareMathOperator{\Re}{Re}

\begin{document}

\title{On the wave representation of hyperbolic, elliptic, and parabolic Eisenstein series}
\author{Jay Jorgenson \and Anna-Maria von Pippich \and Lejla Smajlovi\'{c}
\footnote{The first named author acknowledges grant support from the NSF and PSC-CUNY.
The third named author acknowledges grant support from the Federal Ministry of Education
and Research of Bosnia and Herzegovina.}}
\maketitle

\begin{abstract}\noindent
We develop a unified approach to the construction of the hyperbolic and elliptic Eisenstein series
on a finite volume hyperbolic Riemann surface. Specifically, we derive expressions for the hyperbolic
and elliptic Eisenstein series as integral transforms of the kernel of a wave operator.
Established results in the literature relate the wave kernel to the heat kernel, which admits explicit
construction from various points of view.  Therefore, we obtain a sequence of integral transforms which
begins with the heat kernel, obtains a Poisson and wave kernel, and then yields the
hyperbolic and elliptic Eisenstein series. In the case of a non-compact finite volume hyperbolic Riemann
surface, we finally show how to express the parabolic Eisenstein series in terms of the integral
transform of a wave operator.
\end{abstract}

\vskip .15in
\section{Introduction}
\vskip .10in
\begin{nn}\label{int1}
\textbf{Non-holomorphic Eisenstein series.}
Let $\Gamma\subset\mathrm{PSL}_{2}(\mathbb{R})$ be a Fuchsian subgroup of the first kind
which acts on the hyperbolic space $\Bbb H$ by fractional linear transformations, and
let $M = \Gamma \backslash \Bbb H$ be the finite volume quotient.
One can view $M$ as a finite volume hyperbolic Riemann
surface, possibly with $p_{\Gamma}$ cusps and $e_{\Gamma}$ elliptic fixed points.
Classically, associated to any cusp $p_{j}$ $(j=1,\dots, p_{\Gamma})$ of $M$,
there is a non-holomorphic \textit{parabolic Eisenstein series} defined for
$z\in M$ and $s\in\mathbb{C}$ with $\mathrm{Re}(s)>1$ by
\begin{align}
\mathcal{E}^{\mathrm{par}}_{p_{j}}(z,s)=\sum_{\eta \in \Gamma_{p_{j}}\backslash \Gamma}
                         \Im(\sigma_{p_{j}}^{-1}\eta z)^{s},
\label{def_eis_par}
\end{align}
where $\Gamma_{p_{j}}:=\mathrm{Stab}_{\Gamma}(p_j)=\bigl\langle \gamma_{p_{j}} \bigr\rangle$
is the stabilizer subgroup
generated by a primitive, parabolic element $\gamma_{p_{j}}\in\Gamma$
and
$\sigma_{p_{j}}\in\PSL_{2}(\mathbb{R})$ is the scaling matrix
which satisfies
$\sigma_{p_{j}}^{-1}\gamma_{p_{j}}\sigma_{p_{j}}=
\bigl(\begin{smallmatrix} 1 & 1 \\ 0 & 1 \end{smallmatrix}\bigr).
$
The parabolic Eisenstein series is an automorphic function
on $M$ and admits a meromorphic continuation to the whole
complex $s$-plane with no poles on the line $\mathrm{Re}(s)=1/2$ (see, for example,
\cite{CdV81}, \cite{He83}, \cite{Iwa02}, or \cite{Ku73}).
The parabolic Eisenstein series play a central role in the spectral theory of automorphic functions
on $M$ by contributing the eigenfunctions for the continuous spectrum of the hyperbolic Laplacian on $M$,
and, subsequently, in the Selberg trace formula.

\vskip .10in
In \cite{KM79}, Kudla and Millson defined and studied a form-valued, non-holomorphic Eisenstein series
${\cal E}^{\mathrm{hyp}}_{\gamma,\mathrm{KM}}(z,s)$ associated to any simple
closed geodesic of $M$, or equivalently to any primitive, hyperbolic element
$\gamma\in\Gamma$.
In analogy with results for scalar-valued, parabolic, non-holomorphic Eisenstein series, the following
results were proved in \cite{KM79}.  First, the Eisenstein series ${\cal E}^{\mathrm{hyp}}_{\gamma,\mathrm{KM}}(z,s)$ admits
a meromorphic continuation to all $s\in\mathbb{C}$.  Second, the special value at $s=0$ of the meromorphic
continuation is a harmonic form which is the Poincar\'e dual to the geodesic corresponding to the hyperbolic
element $\gamma\in\Gamma$.
\vskip .10in
A scalar-valued, non-holomorphic, hyperbolic Eisenstein series was
defined in \cite{JKvP10}, and the authors proved that the series admits a meromorphic continuation to all
$s \in \mathbb{C}$.  Let $\gamma \in \Gamma$ be a primitive, hyperbolic element with stabilizer $\Gamma_{\gamma}$.
Let $\mathcal{L}_{\gamma}$ be the geodesic in $\mathbb{H}$ which is invariant by the action of $\gamma$ on $\mathbb{H}$.
Then, for $z\in M$ and $s \in \mathbb{C}$ with $\textrm{Re}(s) > 1$, the
\textit{hyperbolic Eisenstein series} is defined by
\begin{equation*}
{\cal E}^{\mathrm{hyp}}_{\gamma}(z,s) =\sum\limits_{\eta \in \Gamma_{\gamma}\backslash \Gamma}
\cosh(d_{\mathrm{hyp}}(\eta z, \mathcal{L}_{\gamma}))^{-s},
\end{equation*}
where $d_{\mathrm{hyp}}(\eta z, \mathcal{L}_{\gamma})$ denotes the hyperbolic distance from the
point $\eta z$ to the arc $\mathcal{L}_{\gamma}$.  In \cite{JKvP10}, it was proved that ${\cal E}^{\mathrm{hyp}}_{\gamma}(z,s)$
admits an $L^{2}$-spectral expansion, which was explicitly computed, and from which the
meromorphic continuation of ${\cal E}^{\mathrm{hyp}}_{\gamma}(z,s)$ in $s$ was derived.

\vskip .10in
In an unpublished paper from 2004, Jorgenson and Kramer defined an elliptic Eisenstein series
${\cal E}^{\textrm{ell}}_{w}(z,s)$ attached to any point $w$ on $M$.  The problem under study by these authors
was to establish a sup-norm bound for the ratio of the so-called canonical and hyperbolic one-forms
on $\Gamma$; the completion of this work has been published in \cite{JK11}.  For any point $w\in M$,
let $\Gamma_{w}$ denote its stabilizer, which may be trivial.  Then, for $z\in M$, $z\not=w$, and
$s \in \mathbb{C}$ with $\textrm{Re}(s) > 1$, the \textit{elliptic Eisenstein series} is defined by
$$
{\cal E}^{\textrm{ell}}_{w}(z,s) =\sum\limits_{\eta \in \Gamma_{w}\backslash \Gamma}
\sinh(d_{\mathrm{hyp}}(\eta z, w))^{-s}.
$$
In \cite{vP10}, the author studied elliptic Eisenstein series from the point of view of automorphic forms,
ultimately proving the following three main points.  First, the elliptic Eisenstein series ${\cal E}^{\textrm{ell}}_{w}(z,s)$
admits a meromorphic continuation to all $s \in \mathbb{C}$.  Second, the special value at $s=0$ can be expressed
in terms of the norm of a holomorphic modular form which vanishes only at the point $w$.  Finally, although
the elliptic Eisenstein series are not in $L^{2}$, there is a series expansion which expresses
${\cal E}^{\textrm{ell}}_{w}(z,s)$ in a manner similar to the spectral decomposition of an $L^{2}$-function.
This expansion of ${\cal E}^{\textrm{ell}}_{w}(z,s)$ was developed from the
spectral decomposition of the automorphic kernel $K_{s}(z,w)$ which is defined for $z,w\in M$ and
$s\in\mathbb{C}$ with $\Re(s)>1$ by
\begin{equation*}
K_{s}(z,w)=\sum\limits_{\eta \in \Gamma }\cosh (d_{\mathrm{hyp}}(z,\eta w))^{-s}\,.
\end{equation*}

\vskip .10in
In further studies, we refer the interested reader to the articles \cite{Fa07}, \cite{GvP09}
and \cite{GJM08} which prove convergence results amongst the various Eisenstein series when considering
a degenerating sequence of hyperbolic Riemann surfaces.

\end{nn}

\begin{nn}
\textbf{Purpose of this article.} One of the main goals of the present article is to understand
and establish the results of \cite{vP10} from the point of view of integral transforms of
kernel functions, an approach that could be generalized to more general settings.
Specifically, we prove precise relations, or rather integral formulas, which begin with the heat kernel,
proceed to a Poisson kernel and wave kernel, and end with the hyperbolic and the elliptic Eisenstein series.
In the case of hyperbolic Riemann surfaces, which we study in this article, we obtain a precise connection
involving many basic special functions, and we will re-prove some of the theorems mentioned in \ref{int1}.
In addition, if $M$ is non-compact, then we define an automorphic kernel, which is constructed using the
wave kernel, whose zeroth Fourier coefficient is the parabolic Eisenstein series.  As a result, we are
able to express the hyperbolic, elliptic, and parabolic Eisenstein series as integral transforms of the
wave kernel.

\vskip .10in
This article is organized as follows.  In \S 2 we establish notation and provide various background
information needed to make the article as self-contained as possible; in \S 3 we recall results involving
the wave kernel and define an associated wave distribution on a certain class of test functions.  In \S 4 we
express the automorphic kernel $K_{s}(z,w)$ in terms of the wave distribution.  We will prove that both hyperbolic
and elliptic Eisenstein series can be expressed in terms of this automorphic kernel, thus
justifying our assumption that $K_{s}(z,w)$ is indeed the main building block for hyperbolic and
elliptic Eisenstein series.  The connection between $K_{s}(z,w)$ and hyperbolic Eisenstein series is
given in \S 5, and
the connection between the automorphic kernel and the elliptic Eisenstein series is given in \S 6.
Finally, in section 7, we express the hyperbolic Green's function in terms of the of the wave distribution
which yields wave representations of the the parabolic Eisenstein series.

Additionally, we prove that the elliptic Eisenstein series can be realized through the use of a specific test function
in the wave distribution transform.  As a result, we obtain an expression for the elliptic Eisenstein series which can be
readily generalized to settings beyond the study of function theory on finite volume, hyperbolic Riemann surfaces.
Moreover, we derive the meromorphic continuation of the automorphic kernel, the hyperbolic Eisenstein series, and the
elliptic Eisenstein series using the spectral decomposition of the heat kernel at the end of the corresponding sections.
\end{nn}

\begin{nn}
\textbf{Concluding comments.}
In our definition of the wave distribution transform, we use the spectral expansion of $L^2$-functions on $M$ which
requires the meromorphic continuation of the parabolic Eisenstein series.  As a consequence, the logical development
of the results here build upon the continuation of the parabolic Eisenstein series.

Another interesting point in the study of the Eisenstein series considered here is the various forms
of the Kronecker limit formula.  In the subsequent article \cite{JvPS15}, we study the Kronecker limit formula
for elliptic Eisenstein series established in \cite{vP10} and \cite{vP15}.
We prove, among other results, a factorization theorem for holomorphic forms in
terms of elliptic Kronecker limit functions. From this, we obtain a number of explicit evaluations of elliptic
Kronecker limit functions in terms of holomorphic Eisenstein series.
\end{nn}

\begin{nn}
\textbf{Acknowledgements.}
The problems discussed in this article began when the first (J.~J.) and second (A.~v.P.) named authors, together
with J\"urg Kramer, participated in a DAAD-sponsored project at the University of Sarajevo in October 2010.
We all thank M. Avdispahi\'{c} for the opportunity to visit Sarajevo and for his efforts in organizing the
DAAD-sponsored event.  Additionally, all authors benefited greatly from many mathematical conversations
with J. Kramer, and we thank him for sharing his mathematical insight in the writing of the present paper.

\vskip .10in
Finally, the first (J. J.) and third (L. S.) named authors  wish to emphasize the influence of the work
from \cite{vP10}.  The present
article developed from the goal of revisiting the results from this article in order to obtain generalizations
in other settings.  We thank the second named author (A. v.P.) for sharing her results from this
paper.  Going forward, it is important to note that the study
of elliptic Eisenstein series began in \cite{vP10}, and the present article establishes another point of view.
\end{nn}

\vskip .15in
\section{Background material}
\vskip .10in
\begin{nn}\label{2.1}
\textbf{Basic notation.}
As mentioned in the introduction, we let $\Gamma\subset\mathrm{PSL}_{2}(\mathbb{R})$ denote a Fuchsian
group of the first kind acting by fractional
linear transformations on the hyperbolic upper half-plane $\mathbb{H}:=\{z=x+iy\in\mathbb{C}\,
|\,x,y\in\mathbb{R};\,y>0\}$. We let $M:=\Gamma\backslash\mathbb{H}$, which is a finite
volume hyperbolic Riemann surface, and denote by $p:\mathbb{H}\longrightarrow M$
the natural projection. We assume that $M$ has $e_{\Gamma}$
elliptic fixed points and $p_{\Gamma}$ cusps. We identify $M$
locally with its universal cover $\mathbb{H}$.

\vskip .10in
We let $\mu_{\mathrm{hyp}}$ denote the hyperbolic metric on $M$, which is compatible with the
complex structure of $M$, and has constant negative curvature equal to minus one.
The hyperbolic line element $ds^{2}_{\hyp}$, resp.~the hyperbolic Laplacian
$\Delta_{\hyp}$, are given as
\begin{align*}
ds^{2}_{\hyp}:=\frac{dx^{2}+dy^{2}}{y^{2}},\quad\textrm{resp.}
\quad\Delta_{\hyp}:=-y^{2}\left(\frac{\partial^{2}}{\partial
x^{2}}+\frac{\partial^{2}}{\partial y^{2}}\right).
\end{align*}

By $d_{\mathrm{hyp}}(z,w)$ we denote the hyperbolic distance from $z\in\mathbb{H}$ to
$w\in\mathbb{H}$, which satisfies the relation
\begin{align}\label{rel_cosh_u}
\cosh\bigl(d_{\mathrm{hyp}}(z,w)\bigr)= 1+2 u(z,w),
\end{align}
where
\begin{align}
\label{def_u}
u(z,w)=\frac{\left|z-w\right|^{2}}{4\,\mathrm{Im}(z)\mathrm{Im}(w)}\,.
\end{align}
Under the change of coordinates
\begin{align*}
x:=e^{\rho}\cos(\theta)\,,\quad y:=e^{\rho}\sin(\theta)\,,
\end{align*}
the hyperbolic line element, resp.~the hyperbolic Laplacian are
rewritten as
\begin{align*}
ds^{2}_{\hyp}=\frac{d\rho^{2}+d\theta^{2}}{\sin^{2}(\theta)},
\quad\textrm{resp.}\quad\Delta_{\hyp}=-\sin^{2}(\theta)\left
(\frac{\partial^{2}}{\partial\rho^{2}}+\frac{\partial^{2}}{\partial
\theta^{2}}\right).
\end{align*}
For $z=x+iy\in\mathbb{H}$, we define the hyperbolic polar coordinates $\varrho=
\varrho(z),\vartheta=\vartheta(z)$ centered at $i\in\mathbb{H}$ by
\begin{align*}
\varrho(z):=d_{\mathrm{hyp}}(i,z)\,,\quad\vartheta(z):=\measuredangle(\mathcal{L},
T_{z})\,,
\end{align*}
where $\mathcal{L}:=\{z\in\mathbb{H}\,|\,x=\mathrm{Re}(z)=0\}$ denotes the
positive $y$-axis and $T_{z}$ is the euclidean tangent at the unique geodesic
passing through $i$ and $z$ at the point $i$. In terms of the hyperbolic polar
coordinates, the hyperbolic line element, resp.~the hyperbolic Laplacian take
the form
\begin{align*}
ds_{\hyp}^{2}=\sinh^{2}(\varrho)d\vartheta^{2}+d\varrho^{2},
\quad\textrm{resp.}\quad
\Delta_{\hyp}=-\frac{\partial^{2}}{\partial\varrho^{2}}-\frac{1}{\tanh(\varrho)}
\frac{\partial}{\partial\varrho}-\frac{1}{\sinh^{2}(\varrho)}\frac{\partial^{2}}
{\partial\vartheta^{2}}\,.
\end{align*}
\end{nn}

\vskip .10in
\begin{nn}\label{2.2}
\textbf{Hyperbolic and elliptic Eisenstein series.}
In this subsection, we recall the Eisenstein series, which are associated to any
primitive, hyperbolic or elliptic element of $\Gamma$, or equivalently
to any closed geodesic or elliptic fixed point of $M$.

\vskip .10in
First, let $L_{\gamma}$ be the closed geodesic on $M$ in the homotopy class determined by
a hyperbolic element $\gamma\in\Gamma$ with associated stabilizer subgroup $\Gamma_{\gamma}=
\langle\gamma\rangle$. There is a scaling-matrix
$\sigma_{\gamma}\in\mathrm{PSL}_{2}(\mathbb{R})$ such that
\begin{align*}
\sigma_{\gamma}^{-1}\gamma\sigma_{\gamma}=
\begin{pmatrix} e^{\ell_{\gamma}/2}&0\\0&e^{-\ell_{\gamma}/2}\end{pmatrix},
\end{align*}
where $\ell_{\gamma}$ denotes the hyperbolic length of $L_{\gamma}$.
We note that $L_{\gamma}=p(\mathcal{L}_{\gamma})$,
where $\mathcal{L}_{\gamma}:=\sigma_{\gamma}\mathcal{L}$ with $\mathcal{L}$ denoting
the positive $y$-axis.
Using the coordinates $\rho=\rho(z)$ and $\theta=\theta(z)$ introduced in
subsection~\ref{2.1}, the \textit{hyperbolic Eisenstein series associated to the closed geodesic}
$L_{\gamma}$ is defined as a function of $z\in M$ and $s\in\mathbb{C}$ by the series
$$
\mathcal{E}^{\mathrm{hyp}}_{\gamma}(z,s)=
\sum_{\eta\in\Gamma_{\gamma}
\backslash\Gamma}\sin(\theta(\sigma_{\gamma}^{-1}\eta z))^{s}.
$$
From elementary hyperbolic geometry, one can show that the hyperbolic distance
$d_{\mathrm{hyp}}(z,\mathcal{L})$ from $z$ to the geodesic line $\mathcal{L}$ satisfies the formula
\begin{align*}
\sin(\theta(z))\cosh(d_{\mathrm{hyp}}(z,\mathcal{L}))=1\,.
\end{align*}
Therefore, we can re-write the series which defines the hyperbolic Eisenstein series by
\begin{align}
\mathcal{E}^{\mathrm{hyp}}_{\gamma}(z,s)
=\sum_{\eta\in\Gamma_{\gamma}
\backslash\Gamma}\cosh(d_{\mathrm{hyp}}(\eta z,\mathcal{L}_{\gamma}))^{-s}\,.
\label{def_eis_hyp}
\end{align}
Referring to \cite{Fa07}, \cite{GJM08}, or \cite{Ri04}, e.g.,
where detailed proofs are provided, we recall that the series \eqref
{def_eis_hyp} converges absolutely and locally uniformly for any $z\in
\mathbb{H}$ and $s\in\mathbb{C}$ with $\mathrm{Re}(s)>1$, and that the series is
invariant with respect to $\Gamma$. A straightforward computation shows that the hyperbolic
Eisenstein series satisfies the differential equation
\begin{align*}
\big(\Delta_{\mathrm{hyp}}-s(1-s)\big)\mathcal{E}^{\mathrm{hyp}}_{\gamma}(z,s)
=s^{2}\,\mathcal{E}^{\mathrm{hyp}}_{\gamma}(z,s+2)\,.
\end{align*}

\vskip .10in
Now, let $w\in M$ be an arbitrary point with scaling matrix $\sigma_{w}\in\PSL_{2}(\mathbb{R})$
and stabilizer subgroup $\Gamma_{w}=\bigl\langle \gamma_{w} \bigr\rangle$.  The group $\Gamma_{w}$ is trivial
unless $w=e_j$, where $e_{j}$ $(j=1,\dots, e_{\Gamma})$ is an elliptic fixed point of $M$,
and, hence, $\gamma_{w}$ is a primitive, elliptic element of $\Gamma$.
Note that
\begin{align*}
\sigma_{w}^{-1}\gamma_{w}\sigma_{w}=
\begin{pmatrix} \cos(\pi/\mathrm{ord}(w)) & \sin(\pi/\mathrm{ord}(w)) \\
-\sin(\pi /\mathrm{ord}(w)) & \cos(\pi/\mathrm{ord}(w))
\end{pmatrix},
\end{align*}
where $\mathrm{ord}(w)= \mathrm{ord}(\Gamma_{w})$ denotes the order of $w$.
Using the hyperbolic polar coordinates $\varrho=\varrho(z)$ and $\vartheta=
\vartheta(z)$ introduced in subsection~\ref{2.1}, the \emph{elliptic Eisenstein
series} $\mathcal{E}^{\mathrm{ell}}_{w}(z,s)$ \emph{associated to the point}
$w$ is defined by
\begin{align}
\mathcal{E}^{\textrm{ell}}_{w}(z,s)=\sum_{\eta\in\Gamma_{w}
\backslash\Gamma}\sinh(\varrho(\sigma_{w}^{-1}\eta z))^{-s}\,.
\label{def_eis_ell}
\end{align}
Referring to \cite{vP10}, where detailed proofs are provided, we recall that
the series \eqref{def_eis_ell} converges absolutely and locally uniformly for
$z\in M$ with $z\not=w$ and for $s\in\mathbb{C}$ with $\mathrm{Re}(s)>1$, and that
the series is invariant with respect to $\Gamma$. A straightforward computation shows that the
elliptic Eisenstein series satisfies the differential equation
\begin{align*}
\big(\Delta_{\mathrm{hyp}}-s(1-s)\big)\mathcal{E}^{\textrm{ell}}_{w}(z,s)
=-s^{2}\,\mathcal{E}^{\textrm{ell}}_{w}(z,s+2)\,.
\end{align*}
\end{nn}


\vskip .10in
\begin{nn}\label{2.3}
\textbf{Spectral expansions.} Let
$
0=\lambda _{0}<\lambda_1\leq \lambda_2\leq\ldots
$
denote the discrete eigenvalues of the hyperbolic Laplacian
$\Delta_{\mathrm{hyp}}$ acting on smooth functions on $M$; we write
$\lambda_{j}=1/4+t_{j}^{2}=s_{j}(1-s_{j})$, i.e.,
$s_{j}=1/2+it_{j}$ with $t_{j} \geq 0$ or $t_{j}\in(0,i/2]$. When
considering sums over the set of discrete eigenvalues, we will
index the sum by $\lambda_{j} \geq 0$, which allows for the
possibility that the set is finite, as predicted by the
Phillips-Sarnak philosophy, see \cite{PS85}.  Also, if $M$ is
compact, then $p_{\Gamma} = 0$, so sums over the space of
parabolic Eisenstein series do not occur.

\vskip .10in
Under certain hypotheses on a function $f$ on $M$, which are defined carefully in numerous references such as
\cite{Iwa02}, \cite{He83}, or \cite{Ku73}, there is a spectral expansion of $f$ in terms of the eigenfunctions
 $\left\{ \psi _{j}(z)\right\} _{\lambda_{j}\geq 0}$ associated to the
discrete eigenvalues $\lambda_{j}$ of the hyperbolic Laplacian
$\Delta_{\mathrm{hyp}}$ and the parabolic Eisenstein series
$\mathcal{E}^{\mathrm{par}}_{p_{k}}$ $(k=1,\ldots,p_{\Gamma})$
associated to the cusps of $M$; without loss of generality, we
assume that all eigenfunctions of the Laplacian are real-valued.
More precisely, a function $f$ on $M$, under certain assumptions, admits the expansion
\begin{align}
\label{exp_spectral}
f(z)=\sum\limits_{\lambda_{j} \geq 0}\langle f,\psi_{j}\rangle\,\psi_{j}(z)+
\frac{1}{4\pi}\sum\limits_{k=1}^{p_{\Gamma}}\,\int\limits_{-\infty}^{\infty}
\langle f,\mathcal{E}^{\mathrm{par}}_{p_{k}}(\cdot,1/2+ir)\rangle\,
\mathcal{E}^{\mathrm{par}}_{p_{k}}(z,1/2+ir)\,dr\,.
\end{align}
In particular, under certain hypotheses on the point-pair invariant function
$k(z,w)$ on $M\times M$, the automorphic kernel
\begin{equation}
\label{auto_kernel_P-P_In}
K(z,w)=\sum\limits_{\eta \in \Gamma }k(z,\eta w)
\end{equation}
admits the spectral expansion $\eqref{exp_spectral}$ as a function of $z$.
More precisely, let us write
$
k(z,w)=k(u(z,w))=k(u)
$
as a function of $u$ with $u(z,w)$ given by \eqref{def_u}. Suppose that the
Selberg/Harish-Chandra transform $h(r)$ of $k(u)$ exists and satisfies the
following conditions:
\begin{enumerate}
\item[(S1)]
$h(r)$ is an even function;
\item[(S2)]
$h(r)$ is holomorphic in the strip $\left| \mathrm{Im}(r)\right|\leq \frac{1}{2} + \epsilon$ for some $\epsilon>0$;
\item [(S3)]
$h(r) \ll (1+\left|r\right|)^{-2-\delta}$ for some fixed $\delta>0$, as $r\to\infty$
in the domain of condition (S2).
\end{enumerate}
Then, the automorphic kernel \eqref{auto_kernel_P-P_In} admits a
spectral expansion of the form
\begin{align}
\label{exp_spectral_aut_kernel}
K(z,w)=\sum\limits_{\lambda_{j} \geq 0}h(t_j)\psi_{j}(z)\psi _{j}(w)+
\frac{1}{4\pi }\sum\limits_{k=1}^{p_{\Gamma}}\,\int\limits_{-\infty}^{\infty}
h(r)\mathcal{E}_{p_{k}}^{\mathrm{par}}(z,1/2+ir)
\overline{\mathcal{E}_{p_{k}}^{\mathrm{par}}}(w,1/2+ir)dr\,,
\end{align}
which converges absolutely and uniformly on compacta; here, we have
$h(t_j)\psi_{j}(w)=\langle K(\cdot,w),\psi_{j}\rangle$ and
$h(r)\overline{\mathcal{E}_{p_{k}}^{\mathrm{par}}}(w,1/2+ir)=
\langle K(\cdot,w),\mathcal{E}^{\mathrm{par}}_{p_{k}}(\cdot,1/2+ir)\rangle$
(see, e.g., \cite{Iwa02}, Theorem 7.4.).
We recall that the Selberg/Harish-Chandra transform $h(r)$ of $k(u)$ can be computed
in the following three steps:
\begin{align*}
Q(v)=\int\limits_{v}^{\infty}\frac{k(u)}{\sqrt{u-v}}du,\qquad
g(u)=2\,Q\Bigl(\sinh\bigl(\frac{u}{2}\bigr)^2\Bigr), \qquad
h(r)=\int\limits_{-\infty}^{\infty}g(u)e^{iru}du.
\end{align*}
If $h(r)$ satisfies (S1), (S2), and (S3), the Selberg/Harish-Chandra can be inverted as follows:
\begin{align*}
g(u)=\frac{1}{\pi}\int\limits_{0}^{\infty}h(r)\cos(ur)dr,\qquad
Q(v)=\frac{1}{2}g\bigl(2\sinh^{-1}\bigl(\sqrt{v}\bigr)\bigr), \qquad
k(u)=-\frac{1}{\pi}\int\limits_{u}^{\infty}\frac{dQ(v)}{\sqrt{v-u}}.
\end{align*}
For later purpose, we note that, for any function $h:\mathbb{R}\longrightarrow\mathbb{C}$
satisfying (S1) and (S3), the series, resp.~the integral
\begin{equation*}
\sum\limits_{\lambda_{j} \geq 1/4}h(t_j)\,\psi_{j}(z)\psi _{j}(w), \quad \text{ resp. }
\int\limits_{-\infty}^{\infty}
h(r)\,\mathcal{E}_{p_{k}}^{\mathrm{par}}(z,1/2+ir)
\overline{\mathcal{E}_{p_{k}}^{\mathrm{par}}}(w,1/2+ir)dr
\end{equation*}
converges absolutely and uniformly on compacta (see, e.g., \cite{He83}, formula (4.1) on p. 303 with $\chi$ trivial).
Hence, if the function $h:\mathbb{R}\longrightarrow\mathbb{C}$
satisfies condition (S1), the following condition
\begin{enumerate}
\item[(S2')]
$h(r)$ is a well-defined and an even function for $r\in\mathbb{R}\cup[-i/2,i/2]$,
\end{enumerate}
together with condition (S3) in the domain of condition (S2'), that is as $r$ tends
to $\pm\infty$,
then the series and integrals on the right-hand side of \eqref{exp_spectral_aut_kernel}
are well-defined and converge absolutely and uniformly on compacta.
However, these conditions on $h$ do not ensure that
the right-hand side of \eqref{exp_spectral_aut_kernel} represents the spectral expansion of
some automorphic kernel.
Note that we will refer to these conditions simply by writing (S1), (S2'), (S3).
\end{nn}

\vskip .10in
\begin{nn}\label{heat_poisson}
\textbf{The heat kernel and the translated Poisson kernel on $M$.}
The hyperbolic heat kernel $K_{\mathbb{H}}(z,w;t)$ for $z,w\in\mathbb{H}$ and $t\in\mathbb{R}_{>0}$ is
given by the formula
\begin{equation*}
K_{\mathbb{H}}(z,w;t)=\frac{1}{4\pi }\int\limits_{-\infty}^{\infty}
r\tanh (\pi r)e^{-( r^{2}+1/4) t}P_{-\frac{1}{2}
+ir}(\cosh (d_{\mathrm{hyp}}(z,w)))dr,
\end{equation*}
where $P_{\nu}(\cdot)$ denotes the Legendre function of the first kind; see, for example, page 246 of \cite{Ch84}.  The heat kernel is
a fundamental solution associated to the differential operator $\Delta_{\textrm{hyp}} - \partial_{t}$.

\vskip .10in
The translated by $1/4$-Poisson kernel $P_{\mathbb{H},\frac{1}{4}}(z,w;u)$ for
$z,w\in \mathbb{H}$ and $u\in\mathbb{C}$ with $\mathrm{Re}(u)>0$
is given by the $G$-transform of the integral kernel $K_{\mathbb{H}}(z,w;t) e^{t/4}$, namely
\begin{equation*}
P_{\mathbb{H},\frac{1}{4}}(z,w;u)=\frac{u}{\sqrt{4\pi }}
\int\limits_{0}^{\infty}K_{\mathbb{H}}(z,w;t) e^{t/4}e^{-u^2/4t}t^{-1/2}\frac{dt}{t}.
\end{equation*}
The $1/4$-Poisson kernel $P_{\mathbb{H},\frac{1}{4}}(z,w;u)$ is a fundamental solution associated to
the differential operator $\Delta_{\textrm{hyp}} -1/4 - \partial^{2}_{u}$.
By Theorem 5.1 and Remark 5.3 of \cite{JLa03}, the $1/4$-Poisson kernel
$P_{\mathbb{H},\frac{1}{4}}(z,w;u)$ admits a branched meromorphic continuation
to all $u\in\mathbb{C}$ with singularities at $u=\pm i\, d_{\mathrm{hyp}}(z,w)$.

\vskip .10in
The hyperbolic heat kernel $K_{M}(z,w;t)$ on $M$ is a function of $z,w\in M$ and $t\in\mathbb{R}_{>0}$,
and can be obtained by averaging over the group $\Gamma$, namely
\begin{align*}
K_{M}(z,w;t)=\sum\limits_{\eta\in\Gamma}
K_{\mathbb{H}}(z, \eta w;t)\,.
\end{align*}
The spectral expansion \eqref{exp_spectral} of the heat kernel $K_{M}(z,w;t)$  is given by
\begin{align}
\label{spectexp_heat_kernel}
K_{M}(z,w;t)=\sum\limits_{\lambda_{j} \geq 0}e^{-\lambda _{j}t}\psi_{j}(z)\psi _{j}(w)+
\frac{1}{4\pi }\sum\limits_{k=1}^{p_{\Gamma}}\,\int\limits_{-\infty}^{\infty}
e^{-(r^{2}+1/4)t}\mathcal{E}_{p_{k}}^{\mathrm{par}}(z,1/2+ir)
\overline{\mathcal{E}_{p_{k}}^{\mathrm{par}}}(w,1/2+ir)dr\,.
\end{align}

\vskip .10in
For $Z\in\mathbb{C}$ with $\mathrm{Re}(Z)\geq 0$, the translated by $-Z$-Poisson kernel $P_{M,-Z}(z,w;u)$
for $z,w\in M$ and $u\in\mathbb{C}$
with $\mathrm{Re}(u)\geq 0$  is defined by
\begin{equation*}
P_{M,-Z}(z,w;u)=\frac{u}{\sqrt{4\pi }}\int\limits_{0}^{\infty}
K_{M}(z,w;t) e^{-Z t}e^{-u^2/4t}t^{-1/2}\frac{dt}{t}.
\end{equation*}
The function $K_{M}(z,w;t) e^{-Z t}$ is a fundamental solution associated
to the differential ope\-rator $\Delta_{\textrm{hyp}}+Z- \partial_{t}$.
Analogously, the $-Z$-Poisson kernel $P_{M,-Z}(z,w;u)$ is a
fundamental solution associated to the differential operator
$\Delta_{\hyp}+Z-\partial^{2}_{u}$.

\vskip .10in
As before, we write $\lambda_j=1/4+t_j^2$, and divide the sum in
the spectral expansion \eqref{spectexp_heat_kernel} of the heat
kernel on $M$ into the finite sum over $\lambda_j \leq 1/4$, so
then $t_j\in\left(0,i/2\right]$, and the sum over
$\lambda_j\geq1/4$, so then $t_j\in\mathbb{R}_{\geq0}$.
For $\mathrm{Re}(u)>0$ and $\mathrm{Re}(u^2)>0$, using
formula (see \cite{GR07}, formula 3.471.9)
\begin{align*}
\frac{u}{\sqrt{4\pi }}\int\limits_{0}^{\infty}
 e^{-(\lambda _{j}+Z)t}e^{-u^2/4t}t^{-\frac{1}{2}}\frac{dt}{t}=e^{-u\sqrt{\lambda _{j}+Z}},
\end{align*}
where we take the principal branch of the square root,
we get, for $Z\in\mathbb{C}$ with $\mathrm{Re}(Z)\geq 0$,
the spectral expansion
\begin{align*}
P_{M, -Z}(z,w;u) &=
\sum\limits_{\lambda _{j}<\frac{1}{4}}e^{-u\sqrt{\lambda _{j}+Z}} \psi_{j}(z)\psi_{j}(w)
+\sum\limits_{\lambda _{j} \geq \frac{1}{4}}e^{-u \sqrt{\lambda _{j}+Z}} \psi_{j}(z)\psi_{j}(w) \notag \\
&\hspace{4mm}+\frac{1}{4\pi }\sum\limits_{k=1}^{p_{\Gamma}}\,\int\limits_{-\infty}^{\infty}
e^{-u\sqrt{r^{2}+1/4+Z}}
\mathcal{E}_{p_{k}}^{\mathrm{par}}(z,1/2+ir)\overline{\mathcal{E}_{p_{k}}^{\mathrm{par}}}(w,1/2+ir)dr.
\end{align*}
By Theorem 5.2 and Remark 5.3 of \cite{JLa03},
the $-Z$-Poisson kernel $P_{M,-Z}(z,w;u)$ admits an analytic continuation
to $Z=-1/4$. Following the proof of Theorem 5.2
from \cite{JLa03}, we have that the continuation of $P_{M,-Z}(z,w;u)$ to
$Z=-1/4$ with $\mathrm{Re}(u)>0$ and $\mathrm{Re}(u^2)>0$ is given by
\begin{align}\label{formula_poisson}
P_{M,\frac{1}{4}}(z,w;u) &=
\sum\limits_{\lambda _{j}<\frac{1}{4}}e^{-u\sqrt{\lambda _{j}-1/4}}\psi_{j}(z)\psi_{j}(w)+
\sum\limits_{\lambda _{j} \geq \frac{1}{4}}e^{-ut_{j}}\psi_{j}(z)\psi_{j}(w)\notag \\
&\hspace{4mm}+\frac{1}{4\pi }\sum\limits_{k=1}^{p_{\Gamma}}\,\int\limits_{-\infty}^{\infty}
e^{-u\vert r\vert}\mathcal{E}_{p_{k}}^{\mathrm{par}}(z,1/2+ir)\overline{\mathcal{E}_{p_{k}}^{\mathrm{par}}}(w,1/2+ir)dr,
\end{align}
where $\sqrt{\lambda _{j}-1/4}=t_{j}\in \left( 0, i/2 \right] $ is
taken to be the branch of the square root obtained by analytic
continuation through the upper half-plane. Since the function $h(r)=\exp(-u \vert r \vert)$
satisfies the conditions (S1), (S2'), (S3), the series and the integral on the right hand side
of \eqref{formula_poisson} are convergent (see subsection \ref{2.3}).

\vskip .10in
By Theorem 5.2 of \cite{JLa03}, the $1/4$-Poisson kernel $P_{M,\frac{1}{4}}(z,w;u)$
admits a branched meromorphic
continuation to all $u\in\mathbb{C}$ with branched singularities at all points
of the form $u=\pm i \rho$ whenever $\rho$ is the length of a geodesic path from
$z$ to $w$, and possibly at $u=0$. This result is proved in greater generality using the
language of pseudo-differential
operators in \cite{DG75}.

\end{nn}

\vskip .15in
\section{The wave distribution}
\vskip .10in
Choosing a branch of the meromorphic continuation of the $1/4$-Poisson kernel
$P_{M,\frac{1}{4}}(z,w;u)$ to all $u\in\mathbb{C}$,
the functions $P_{M,\frac{1}{4}}(z,w;\pm iu)$ represent a fundamental solution of
the translated wave equation $\Delta_{\hyp}-1/4+\partial^{2}_{u}$.
For $z,w \in M$, $z \neq w$, and $u\in\mathbb{R}$, we define the translated by
$\frac{1}{4}$ wave kernel on $M$ by
$$
W_{M,\frac{1}{4}}(z,w;u):=
P_{M,\frac{1}{4}}(z,w;iu)+P_{M,\frac{1}{4}}(z,w;-iu).
$$
Because of convergence issues, the meromorphic continuation may not be obtained by simply
replacing $u$ by $-iu$ (or $iu$) in the formula \eqref{formula_poisson} for the Poisson kernel on $M$.
Therefore, we will introduce the wave distribution $\mathcal{W}_{M,\frac{1}{4}}(z,w)$ using the exact
expression \eqref{formula_poisson} for the analytic continuation of $P_{M,\frac{1}{4}}(z,w;u)$.
Loosely speaking, we may say that the wave distribution is defined in such a way that in
\eqref{formula_poisson}, the function $\exp(-ur)$ is replaced by
the distribution associated to $\exp(-iur)$, or equivalently, by the distribution associated to $\cos(ur)$
when restricting our attention to even functions of $r$.

\vskip .10in
Let $\mathbb{R}^{+}$ denote the set of non-negative
real numbers. By $C^{\infty }(\mathbb{R}^{+})$ we mean the class
of all infinitely differentiable functions on $\Bbb R^{+}$ and by
$C^{\infty }_{0}(\mathbb{R}^{+})\subset C^{\infty
}(\mathbb{R}^{+})$ we denote the subclass of functions with
compact support on $\Bbb R^{+}$.  As usual, $g^{(j)}$ denotes the
$j$-th derivative of the function $g$ ($j\geq1$) and the Schwartz
space of functions are those functions which are infinitely
differentiable and rapidly decreasing.

\vskip .10in
\begin{definition}\rm
For $a\in\mathbb{R}^{+}$, let $S^{\prime}(\mathbb{R}^{+},a)$ denote the subspace of functions
$g$ in the Schwartz space on $\mathbb{R}^{+}$ with $g^{\prime
}(0)=0$ and such that $\left\vert g(u)\right\vert \exp (u a)$ is dominated
by an integrable function $G(u)$ on $\mathbb{R}^{+}$.
\end{definition}

\vskip .10in
\begin{definition}\label{def_H}\rm
Let $a\in\mathbb{R}^{+}$ and  $g\in S^{\prime}(\mathbb{R}^{+},a)$. For $r\in\mathbb{C}$,
we formally define the function
\begin{align} \label{integral_B}
H(r,g):=2\,\int\limits_{0}^{\infty}\cos (ur)g(u)du.
\end{align}
\end{definition}

\vskip .10in
\begin{lemma}\label{lemma_integral_A_B}
\begin{enumerate}
\item[(i)]
Let $g\in S^{\prime }(\mathbb{R}^{+},1/2)$ be such that the first
three derivatives of $g$ are integrable and have a limit as $u\to\infty$. Then, the function $H(r,g)$ is well-defined
and satisfies the conditions (S1), (S2'), and (S3)  of subsection \ref{2.3} with $\delta=1$.
\item[(ii)]
Let $\eta>0$ and let $g\in S^{\prime }(\mathbb{R}^{+},1/2+\eta)$ be such that
$g^{(j)}(u) \exp (u(1/2+\eta))$ is bounded by some integrable function on
$\mathbb{R}^{+}$ for $j=1,2,3$. Then, the function $H(r,g)$ satisfies the conditions (S1), (S2),
and (S3) of subsection \ref{2.3} for any $0<\epsilon < \eta$ and with $\delta=1$.
\end{enumerate}
\end{lemma}

\begin{proof}
\emph{Part (i):} Trivially, for $u \in\mathbb{R}^{+}$ and $r\in \mathbb{C}$ with
$\left| \mathrm{Im}(r)\right| \leq 1/2$, the function $\cos (ur)$ is dominated by $\exp(u/2)$.
Since $g\in S^{\prime }(\mathbb{R}^{+},1/2)$ by assumption, the function $\left\vert g(u)\right\vert \exp (u/2)$
is dominated by an integrable function $G(u)$ on $\mathbb{R}^{+}$.  Therefore, the integral \eqref{integral_B}
defining $H(r,g)$ exists and, hence, the function $H(r,g)$ is well-defined for $r\in \mathbb{C}$ with
$\left| \mathrm{Im}(r)\right| \leq 1/2$ and represents an even function of $r$.
Further, when $r$ is real, integrating by parts three times, using the fact that $g'(0)=0$ and that first
three derivatives of $g$ are integrable and have a limit as $u\to \infty$, we get that $H(r,g)$
satisfies the asserted bound of condition (S3).

\vskip .10in
\emph{Part (ii):} For $u \in\mathbb{R}^{+}$, the
function $g(u)\cos(u r)$ is uniformly bounded in the strip $\left|\mathrm{Im}(r)\right| \leq 1/2 + \epsilon$
by $G(u) \exp (-(\eta-\epsilon)u)$ for some integrable function $G(u)$ on $\mathbb{R}^{+}$. Therefore,
the integral \eqref{integral_B} defining $H(r,g)$ converges absolutely and uniformly in $r$ on every
compact subset of the strip $\left| \mathrm{Im}(r)\right| \leq 1/2 + \epsilon$, thus
representing a holomorphic function that is obviously even. This proves that the function $H(r,g)$ satisfies
conditions (S1) and (S2). In order to prove (S3), we apply integration by parts three
times, now using the fact that $g'(0)=0$ and that $\cos
(ur)g^{(j)}(u)$ and $\sin (ur)g^{(j)}(u)$ are uniformly bounded in $r$ by $G_{j}(u) \exp (-(\eta -\epsilon)u)$
for some integrable functions $G_{j}$ on $\mathbb {R}^{+}$ ($j=1,2,3$).
\end{proof}

\vskip .10in
\begin{definition} \label{WaveDistrDef}\rm
Let $z, w\in M$. For $g\in C^{\infty }_{0}(\mathbb{R}^{+})$, the
wave distribution $\mathcal{W}_{M,\frac{1}{4}}(z,w)(g)$ applied to
the function $g$ is defined by
\begin{align}
\mathcal{W}_{M,\frac{1}{4}}(z,w)(g)&=
\sum\limits_{\lambda _{j} \geq 0} H(t_j,g) \,\psi _{j}(z)\psi _{j}(w) \notag\\
&\hspace{4mm}+\frac{1}{4\pi }\sum\limits_{k=1}^{p_{\Gamma }}\,\int\limits_{-\infty}^{\infty} H(r,g)\,
\mathcal{E}_{p_{k}}^{\mathrm{par}}(z,1/2+ir)
\overline{\mathcal{E}_{p_{k} }^{\mathrm{par}}}(w,1/2+ir)
dr \label{KernelDistr},
\end{align}
where  $\sqrt{\lambda _{j}-1/4}=t_{j}$ (in the case when
$\lambda_{j} < 1/4$ we take $t_{j} \in (0,i/2]$) and the coefficient $H(\cdot,g)$
is given by \eqref{integral_B}.
\end{definition}

\vskip .10in
\begin{proposition}\label{Wavedistr}
\begin{enumerate}
\item[(i)] Let $z, w\in M$.
Let $g\in S^{\prime }(\mathbb{R}^{+},1/2)$ be such that the first
three derivatives of $g$ are integrable and have a limit as $u\to\infty$.
Then, the wave distribution
$\mathcal{W}_{M,\frac{1}{4}}(z,w)(g)$ is well-defined, meaning the
integral and series in \eqref{KernelDistr} are convergent.
\item[(ii)] Let $z, w\in M$. Let $\eta>0$ and let $g\in S^{\prime }(\mathbb{R}^{+},1/2+\eta)$ be such that
$g^{(j)}(u) \exp ((1/2+\eta)u)$ is bounded by some integrable function on
$\mathbb{R}^{+}$ for $j=1,2,3$. Then, the wave distribution represents the spectral expansion of
an automorphic kernel, namely we have
\begin{align}\label{WaveAsAutKernel}
\mathcal{W}_{M,\frac{1}{4}}(z,w)(g)=\sum_{\eta\in\Gamma}k(z,\eta w),
\end{align}
where $k(z,w)=k(u(z,w))=k(u)$ is the inverse of the Selberg/Harish-Chandra
transform of $H(r,g)$. Moreover, the series in \eqref{WaveAsAutKernel} converges
absolutely and uniformly on compacta.
\end{enumerate}
\end{proposition}

\begin{proof}
\emph{Part (i):}
By part (i) of Lemma \ref{lemma_integral_A_B}, $H(r,g)$ is
a well-defined and even function for $r\in[-i/2,i/2]$. Hence, the finite sum
\begin{align*}
\sum\limits_{\lambda_{j} < 1/4}H(t_j,g) \,\psi_{j}(z)\psi _{j}(w)
\end{align*}
is well-defined. Further, for $r\in\mathbb{R}$, the function $H(r,g)$
satisfies property (S3) with $\delta=1$. As recalled in subsection \ref{2.3},
this ensures that the series and the integral on right hand side of \eqref{KernelDistr}
\begin{align*}
\sum\limits_{\lambda_{j} \geq 1/4}H(t_j,g) \,\psi_{j}(z)\psi _{j}(w), \quad
\int\limits_{-\infty}^{\infty}
H(r,g) \,\mathcal{E}_{p_{k}}^{\mathrm{par}}(z,1/2+ir)
\overline{\mathcal{E}_{p_{k}}^{\mathrm{par}}}(w,1/2+ir)dr
\end{align*}
converge absolutely, which completes the proof of part (i).

\vskip .10in
\emph{Part (ii):}
By part (ii) of Lemma \ref{lemma_integral_A_B}, $H(r,g)$ satisfies the
(S1), (S2), and (S3). As recalled in subsection \ref{2.3}, this ensures that
the automorphic kernel
$$
\sum_{\eta\in\Gamma}k(z,\eta w)
$$
admits a spectral expansion of the form \eqref{exp_spectral_aut_kernel},
which converges absolutely and uniformly on compacta and
which equals the right hand side of \eqref{KernelDistr}, since $h(r)=H(r,g)$
is the Selberg/Harish-Chandra transform of $k(u)=k(u(z,w))=k(z,w)$.
This completes the proof.
\end{proof}

\vskip .10in
\begin{proposition}\label{uniquness}
Let $z, w\in M$ with $z \neq w$.
Then, there exists a continuous, real-valued
function $F(z,w;u)$ of $u\in\mathbb{R}^{+}$ with the following properties:
\begin{enumerate}
\item[(i)] $F(z,w;u)=-\sum\limits_{\lambda _{j}<\frac{1}{4}}
e^{u\sqrt{1/4-\lambda _{j}}}\, \left(\sqrt{1/4-\lambda _{j}}\right)^{-3}\,\psi _{j}(z)\psi _{j}(w) + O(u^3)$\,\, as $u \rightarrow \infty$.
\item[(ii)] $F^{(j)}(z,w;u) = O(u^{3-j})$\,\, as $u \rightarrow 0$ ($j=0,1,2$).
\item[(iii)] For any $g\in S^{\prime }(\mathbb{R}^{+},1/2)$ such that
$g^{(j)}(u) \exp (u/2)$ has a limit as $u\to\infty$ and is bounded by some integrable function
on $\mathbb{R}^{+}$ for $j=0,1,2,3$,
we have
\begin{equation} \label{Wave via Poisson}
\mathcal{W}_{M,\frac{1}{4}}(z,w)(g)=
\int\limits_{0}^{\infty}F(z,w;u) g^{(3)}(u)du.
\end{equation}
\end{enumerate}
\end{proposition}

\begin{proof}
Let $z, w\in M$ with $z \neq w$. For $\zeta\in\mathbb{C}$ with $\Re(\zeta)\geq 0$, we define
\begin{align}\notag
\widetilde{F}_{M,\frac{1}{4}}(z,w;\zeta)=&
\sum\limits_{\lambda _{j}<\frac{1}{4}}f(t_{j},\zeta)\,\psi _{j}(z)\psi _{j}(w)
+\sum\limits_{\lambda _{j}\geq\frac{1}{4}} f( t_{j},\zeta)\, \psi _{j}(z)\psi _{j}(w) \\
&+\frac{1}{4\pi }\sum\limits_{k=1}^{p_{\Gamma }}
\int\limits_{-\infty}^{\infty }f( \vert r\vert,\zeta  )\,
\mathcal{E}_{p_{k}}^{\mathrm{par}}(z,1/2+ir)
\overline{\mathcal{E}_{p_{k}}^{\mathrm{par}}}(w,1/2+ir)dr, \label{formula_poisson3}
\end{align}
where  $\sqrt{\lambda _{j}-1/4}=t_{j}$ (in the case when
$\lambda_{j} < 1/4$ we take $t_{j} \in (0,i/2]$) and
where we have set
$$
f(r,\zeta) = \frac{\exp(-\zeta r) -1 + \zeta \sin(r) -(\zeta^{2}/2)\sin(r)^{2}}{(-r)^{3}}
$$
for $r \neq 0$ in the strip $\left|\mathrm{Im}(r)\right|\leq1/2$, and
$
f(0,\zeta)=(\zeta^{3}+\zeta)/6,
$
by continuation. Observe that, for $\Re(\zeta)\geq 0$ and $r\in\mathbb{R}^{+}$, we have
$$
f(r,\zeta) = O_{\zeta}(1) \,\,\textrm{ as $r\rightarrow 0$ } \qquad \textrm{and} \qquad
f(r,\zeta) = O_{\zeta}(r^{-3}) \,\,\textrm{ as $r\rightarrow +\infty$.}
$$
Therefore, the spectral series \eqref{formula_poisson3} converges absolutely
and uniformly, provided $\Re (\zeta) \geq 0$, by the statements in subsection \ref{2.3}.

\vskip .10in
Similarly,
the spectral series obtained by replacing $f(r,\zeta)$ in \eqref{formula_poisson3} with any of its
first three derivatives (with respect to the variable $\zeta$) is absolutely and uniformly convergent,
provided $\Re(\zeta) > 0$.  Therefore, term by term differentiation is valid, and for $\Re(\zeta) > 0$
one has, using $d^{3}f(r,\zeta)/d\zeta^3 =\exp(-\zeta r)$ and by comparing with \eqref{formula_poisson},
the identity
\begin{align}\label{thirdderivativeFtildeequalsP}
\frac{d^{3}}{d\zeta^{3}}\widetilde{F}_{M,\frac{1}{4}}(z,w;\zeta) = P_{M,\frac{1}{4}}(z,w;\zeta).
\end{align}
Now, let
$
P^{(0)}_{M,\frac{1}{4}}(z,w;\zeta) = P_{M,\frac{1}{4}}(z,w;\zeta)
$
and, for $k\in\mathbb{N}, k \geq1$, define
$$
P^{(k)}_{M,\frac{1}{4}}(z,w;\zeta) = \int\limits_{0}^{\zeta}P^{(k-1)}_{M,\frac{1}{4}}(z,w;\xi)d\xi,
$$
where the integral is taken along the ray $\left[0,\zeta\right)$ inside the half plane $\Re(\zeta) > 0$.
With this notation, we have shown that
\begin{align}\label{thirdint}
\widetilde{F}_{M,\frac{1}{4}}(z,w;\zeta)+ q(z,w;\zeta) = P^{(3)}_{M,\frac{1}{4}}(z,w;\zeta),
\end{align}
where $q(z,w;\zeta)$ is a degree two polynomial in $\zeta$ with coefficients which depend on $z$ and $w$.
For $z \neq w$, the function $P^{(0)}_{M,\frac{1}{4}}(z,w;\zeta)$ has a limit as $\zeta$ approaches zero;
therefore, we get the bound
\begin{align}\label{Pkbd}
P^{(k)}_{M,\frac{1}{4}}(z,w;\zeta) = O(\zeta^{k})   \textrm{ as $\zeta\to 0$. }
\end{align}
With all this, we define, for $z, w\in M$ with $z \neq w$ and $u \in \Bbb R^{+}$, the function
$$
F(z,w;u) = \frac{1}{i}\left(\bigl(\widetilde{F}_{M,\frac{1}{4}}(z,w;iu) + q(z,w;iu)\bigr) -
\bigl(\widetilde{F}_{M,\frac{1}{4}}(z,w;-iu) + q(z,w;-iu)\bigr)\right).
$$
The function $F(z,w;u)$ can be expressed using the spectral expansion \eqref{formula_poisson3},
from which property (i) is proved, employing the equality
$$
\frac{1}{i}\bigl(f(r,iu)-f(r,-iu)\bigr)=\frac{2}{r^3}\bigl(\sin(ur)-u \sin(r)\bigr)
$$
for $r\not=0$ in the strip $\left|\mathrm{Im}(r)\right|\leq1/2$,
and using that $q(z,w;iu)-q(z,w;-iu)$ is a degree one polynomial in $u$ with constant coefficient equal to zero.
Assertion (ii) is proved directly from \eqref{thirdint} using the bound \eqref{Pkbd}.

\vskip .10in
Finally, one proves (iii) as follows. First, from integration by parts, we have that
$$
\int\limits_{0}^{\infty}\bigl(q(z,w;iu) -q(z,w;-iu)\bigr)g^{(3)}(u)du=0.
$$
Second, the spectral expansion \eqref{formula_poisson3} converges absolutely and
uniformly for $\Re(\zeta) = 0$, and the asymptotic expansion asserted in (i) holds.  Therefore, we get
\begin{align*}
&\frac{1}{i}\int\limits_{0}^{\infty}\,\left(
\widetilde{F}_{M,\frac{1}{4}}(z,w;iu) -\widetilde{F}_{M,\frac{1}{4}}(z,w;-iu)\right)g^{(3)}(u)du = \\
&
\sum\limits_{\lambda _{j}\geq0} G(t_{j},g)\, \psi _{j}(z)\psi _{j}(w)
 +\frac{1}{4\pi }\overset{p_{\Gamma }}{\underset{k=1}{\sum }}\underset{-\infty}{
\overset{\infty }{\int }}G(\vert r\vert, g)\,
\mathcal{E}_{p_{k}}^{\mathrm{par}}(z,1/2+ir)
\overline{\mathcal{E}_{p_{k}}^{\mathrm{par}}}(w,1/2+ir)dr, \notag
\end{align*}
where
$$
G(r,g) =
\frac{1}{i}\int\limits_{0}^{\infty}\bigl(f(r,iu)-f(r,-iu)\bigr)
g^{(3)}(u)du.
$$
Integrating by parts three times, we finally derive $G(r,g) = H(r,g)$, which completes the
proof.
\end{proof}

\vskip .15in
\section{The wave representation of an automorphic kernel}\label{4}
\vskip .10in
In this section, we study the following well-known automorphic kernel and
we express it in terms of the wave distribution.

\vskip .10in
For $z,w\in M$ and $s\in\mathbb{C}$ with $\Re(s)>1$, we formally
define the automorphic kernel $K_{s}(z,w)$ by the series
\begin{align}\label{auto_kernel_series}
K_{s}(z,w)=\sum\limits_{\eta \in \Gamma }\cosh (d_{\mathrm{hyp}}(z,\eta w))^{-s}.
\end{align}
Elementary considerations show that the series which defines $K_{s}(z,w)$ converges
absolutely for all $s\in\mathbb{C}$ with $\mathrm{Re}(s) > 1$, uniformly
when $z$ and $w$ are restricted to any compact subset of $M$.

\vskip .10in
Furthermore, $K_{s}(z,w)$ is bounded on $M$ as a function of $z$ and, hence, belongs to
$L^{2}(M)$ and admits a spectral expansion (see, e.g., \cite{vP10}, Proposition 5.1.1, for an
explicit computation).
In order to express the automorphic kernel in terms of the action of the wave distribution,
we will outline a different proof of the spectral expansion following the ideas of Selberg.

\vskip .10in
\begin{proposition}\label{SExpK}
For $z,w\in M$ and $s\in\mathbb{C}$ with $\Re(s)>1$, the automorphic kernel
$K_s(z,w)$ admits the spectral expansion
\begin{align}
K_{s}(z,w)=\sum\limits_{\lambda_{j} \geq 0}a_{1/2+it_j}(s)\,\psi _{j}(z)\psi _{j}(w)+
\frac{1}{4\pi }\overset{p_{\Gamma }}{\underset{k=1}{\sum }}\underset{-\infty
}{\overset{\infty }{\int }}a_{1/2+ir}(s)\,
\mathcal{E}_{p_{k}}^{\mathrm{par}}(z,1/2+ir)
\overline{\mathcal{E}_{p_{k}}^{\mathrm{par}}}(w,1/2+ir)dr,
\label{SpectralExpP}
\end{align}
which converges absolutely and uniformly on compacta;
here, we have set
\begin{align*}
a_{\nu}(s):=\frac{2^{s-1}\sqrt{\pi }}{\Gamma (s)}\,
\Gamma\Bigl(\frac{s-\nu}{2}\Bigr)\Gamma\Bigl(\frac{s-1+\nu}{2}\Bigr).
\end{align*}
\end{proposition}

\begin{proof}
Let  $z,w\in M$ and $s\in\mathbb{C}$ with $\Re(s)>1$.
First, we compute the Selberg/Harish-Chandra transform
of the point-pair invariant function $k_{s}(z,w):=\cosh(d_{\mathrm{hyp}}(z, w))^{-s}$
 by applying the three steps recalled in subsection \ref{2.3}. To do this, we write
$k_{s}(z,w)$ as a function of $u=u(z,w)$ using relation \eqref{rel_cosh_u}, namely
$k_{s}(u(z,w))=k_{s}(u)=(1+2 u)^{-s}.$ Hence, in the first step, we have to compute
\begin{align*}
Q_{s}(v):=\int\limits_{v}^{\infty}\frac{k_{s}(u)}{\sqrt{u-v}}du
=\int\limits_{0}^{\infty}\frac{(1+2 v+2u)^{-s}}{\sqrt{u}}du
=(1+2v)^{1/2-s}\,\frac{2^{-1/2}\sqrt{\pi}\,\Gamma(s-\frac{1}{2})}{\Gamma(s)};
\end{align*}
here, for the last equality we used \cite{GR07}, formula 3.251.11 with $\beta:=2/(1+2v)$, $\mu:=1/2$, $p=1$ and $\nu:=s$. In the second step, we get
\begin{align*}
g_{s}(u):=2\,Q_{s}\bigl(\sinh\bigl(\frac{u}{2}\bigr)^2\bigr)=
\cosh(u)^{1/2-s}\,\frac{2^{1/2}\sqrt{\pi}\,\Gamma(s-\frac{1}{2})}{\Gamma(s)}.
\end{align*}
In the last step, we compute
\begin{align}\label{hsH}
h_{s}(r):=\int\limits_{-\infty}^{\infty}g_s(u)e^{iru}du=\frac{2^{3/2}\sqrt{\pi}\,\Gamma(s-\frac{1}{2})}{\Gamma(s)}\,\int
\limits_{0}^{\infty}\cos (ur)\cosh(u)^{-(s-1/2)}du= H(r,g_s),
\end{align}
where $H(r,g_s)$ is defined by \eqref{integral_B}. Now, in the case
$r\in\mathbb{R}$, $r\not=0$, the case $r\in[-i/2,i/2]$, $r\not=0$, resp.~the case $r=0$, we have
(see \cite{GR07}, formula 3.985.1, formula 3.512.1, formula 3.512.2, respectively)
\begin{align}\label{f1}
\int\limits_{0}^{\infty}\cos (ur)\cosh(u)^{-\nu}du=
\frac{2^{\nu-2}}{\Gamma(\nu)}\,\Gamma\Bigl(\frac{\nu-ir}{2}\Bigr)\Gamma\Bigl(\frac{\nu+ir}{2}\Bigr),
\end{align}
where $\Re(\nu)>1/2$.
Substituting \eqref{f1} with $\nu:=s-1/2$ into \eqref{hsH}, we immediately get
for $r\in\mathbb{R}$ or $r\in[-i/2,i/2]$ the identity
\begin{align*}
h_{s}(r)=\frac{2^{s-1}\sqrt{\pi }}{\Gamma (s)}\,
\Gamma\Bigl(\frac{s-\frac{1}{2}-ir}{2}\Bigr)\Gamma\Bigl(\frac{s-\frac{1}{2}+ir}{2}\Bigr).
\end{align*}
Now, let $\eta=(\mathrm{Re}(s) - 1)/2$. Then, $1/2+\eta=\mathrm{Re}(s)/2$, and,
obviously, $g_s(u)\in S^{\prime }(\mathbb{R}^{+},1/2+\eta)$.
Moreover, $g_{s}^{(j)}(u) \exp (u(1/2+\eta))$ is bounded by some integrable function on
$\mathbb{R}^{+}$ for $j=1,2,3$.
Hence, by part (ii) of lemma \ref{lemma_integral_A_B}, the function $H(r,g_s)=h_{s}(r)$ satisfies the conditions (S1), (S2),
and (S3) of subsection \ref{2.3} for any $0<\epsilon < \eta$ and with $\delta=1$.
Therefore, $K_{s}(z,w)$ admits the spectral expansion
\begin{equation*}
K_{s}(z,w)=\sum\limits_{\lambda_{j} \geq 0}h_{s}(t_{j})\,\psi
_{j}(z)\psi _{j}(w)+ \frac{1}{4\pi }\overset{p_{\Gamma
}}{\underset{k=1}{\sum }}\underset{-\infty }{\overset{\infty
}{\int }}h_{s}(r)\, \mathcal{E}_{p_{k}}^{\mathrm{par}}(z,1/2+ir)
\overline{\mathcal{E}_{p_{k}}^{\mathrm{par}}}(w,1/2+ir)dr,
\end{equation*}
which converges absolutely and uniformly on compacta, as asserted.
\end{proof}

A direct consequence of the proof of Proposition \ref{SExpK} is that
the automorphic kernel $K_{s}(z,w)$ may be represented via an
action of the wave distribution.

\vskip .10in
\begin{theorem}\label{PoencViaWave}
For $z,w\in M$ and $s\in\mathbb{C}$ with $\Re(s)>1$, we
have the representation
\begin{equation}
K_{s}(z,w)=\mathcal{W}_{M,\frac{1}{4}}(z,w)(g_{s}),
\label{RepPoencWiaWave}
\end{equation}
which converges absolutely and uniformly on compacta; here, we have set
\begin{align}\label{g_s}
g_{s}(u)= \frac{2^{1/2}\sqrt{\pi}\,\Gamma(s-\frac{1}{2})}{\Gamma(s)}\, \cosh(u)^{-(s-1/2)}.
\end{align}
\end{theorem}

\begin{proof}
Let $z, w\in M$ and let $\eta=(\mathrm{Re}(s) - 1)/2$. Then, we have
$g_s\in S^{\prime }(\mathbb{R}^{+},1/2+\eta)$ and $g_s^{(j)}(u) \exp (u(1/2+\eta))$ is
bounded by some integrable function on $\mathbb{R}^{+}$ for $j=1,2,3$.
Hence, by Proposition \ref{Wavedistr}, we have
\begin{align*}
\mathcal{W}_{M,\frac{1}{4}}(z,w)(g_s)=\sum_{\eta\in\Gamma}k_s(z,\eta w)=K_s(z,w),
\end{align*}
which converges absolutely and uniformly on compacta;
here, $k_s(u)=k_s(u(z,w))=k_s(z,w)=\cosh(d_{\mathrm{hyp}}(z, w))^{-s}$
is the inverse of the Selberg/Harish-Chandra transform of $H(r,g_s)$,
as the proof of Proposition \ref{SExpK} shows.
\end{proof}

We conclude this section by proving the meromorphic continuation of the automorphic kernel,
which is known in the literature (see, for example, \cite{vP10}).
However, the proof we present here is given within the framework of the wave distribution, hence
provides a point of view which can be generalized to other settings.

\vskip .10in
\begin{lemma}
For $s\in\mathbb{C}$ with $\mathrm{Re}(s)>1$, $n\in\mathbb{N}$, and $r\in\mathbb{R}\cup[-i/2,i/2]$,
we have
\begin{align}\label{formula_shift_H}
H(r,g_{s})=
\frac{2^{-2n}\,(s)_{2n}}
{(\frac{s}{2}-\frac{1}{4}-\frac{ir}{2})_{n} (\frac{s}{2}-\frac{1}{4}+\frac{ir}{2})_{n}}
\,H(r,g_{s+2n}),
\end{align}
where $(\cdot)_{n}$ denotes the Pochhammer symbol,
 $H(r,\cdot)$ is defined by \eqref{integral_B}, and $g_s$ is given by \eqref{g_s}.
\end{lemma}

\begin{proof}
Let $n\in\mathbb{N}$ and $s\in\mathbb{C}$ with $\mathrm{Re}(s)>1$.
By definitions \eqref{integral_B} and \eqref{g_s}, we have
\begin{align*}
H(r,g_{\alpha})=
\frac{2^{3/2}\sqrt{\pi}\,\Gamma(\alpha-\frac{1}{2})}{\Gamma(\alpha)}\,
\int\limits_{0}^{\infty}\cos(u r) \cosh(u)^{-(\alpha-1/2)} du.
\end{align*}
Hence, for $r\in\mathbb{R}\cup[-i/2,i/2]$,
using formula \eqref{f1} with $\nu:=\alpha-1/2$ for $\alpha=s$ and $\alpha=s+2n$,
we derive the identity \eqref{formula_shift_H}.
\end{proof}

\vskip .10in
\begin{theorem}\label{Poen_cont}
For any $z,w\in M$, the automorphic kernel $K_{s}(z,w)$ admits a meromorphic continuation
to the whole complex $s$-plane. The possible poles of the function
$\Gamma(s)\Gamma(s-1/2)^{-1}K_{s}(z,w)$ are located at the
points $s=1/2\pm i t_{j}-2n$, where $n\in\mathbb{N}$ and $\lambda _{j}=1/4+t_j^2$ is a discrete
eigenvalue of $\Delta_{\hyp}$, at the points $s=1-\rho-2n$, where $n\in\mathbb{N}$
and $\rho\in\left(1/2,1\right] $ is a pole of $\mathcal{E}_{p_{k}}^{\mathrm{par}}(z,s)$,
and at the points $s=\rho -2n$, where $n\in\mathbb{N}$ and $\rho$ is a pole of
$\mathcal{E}_{p_{k}}^{\mathrm{par}}(z,s)$ with  $\mathrm{Re}(\rho) <1/2$.
\end{theorem}

\begin{proof}
We start by proving that $K_{s}(z,w)$ has a meromorphic continuation to the halfplane
$\Re(s)>1-2n$ for any $n\in\mathbb{N}$. Using the wave representation \eqref{RepPoencWiaWave}
of $K_{s}(z,w)$ and substituting formula \eqref{formula_shift_H} for the coefficients
$H(t_j,g_{s})$, we get for $s\in\mathbb{C}$ with $\mathrm{Re}(s)>1$ the identity
\begin{align}\label{waverep_K}
\frac{2^{2n}\,\Gamma(s)}{\Gamma(s+2n)}\, K_{s}(z,w)&=
\sum\limits_{\lambda _{j} \geq 0}\frac{H(t_j,g_{s+2n})}{h_{n}(t_j,s)} \,\psi _{j}(z)\psi _{j}(w)
\notag\\
&\hspace{4mm}
+\frac{1}{4\pi }\sum\limits_{k=1}^{p_{\Gamma }}\,\int\limits_{-\infty}^{\infty}
\frac{H(r,g_{s+2n})}{h_{n}(r,s)}\,
\mathcal{E}_{p_{k}}^{\mathrm{par}}(z,1/2+ir)
\overline{\mathcal{E}_{p_{k} }^{\mathrm{par}}}(w,1/2+ir)
dr,
\end{align}
where
$
h_{n}(r,s):=(\frac{s}{2}-\frac{1}{4}-\frac{ir}{2})_{n} (\frac{s}{2}-\frac{1}{4}+\frac{ir}{2})_{n}
$
and $\lambda _{j}=1/4+t_j^2$.
%
Since $(a)_n=\prod_{j=0}^{n-1}(a+j)$, we have $h_{n}(r,s)  \sim r^{2n}$, as $r\to\infty$.
This proves that the series in \eqref{waverep_K} arising from the
discrete spectrum is locally absolutely and uniformly convergent as
a function of $s\in\mathbb{C}$ for $\Re(s)>1-2n$ away from the poles
of $h_{n}(t_j,s)^{-1}$. The location of the poles arising from this part is now straightforward
referring to the fact that zeros of $h_{n}(t_j,s)$ for $\Re(s)>1-2n$ are at the points
$s=1/2\pm i t_{j}-2\ell$ for $\ell=0,\ldots,n-1$.

\vskip .10in
To prove the meromorphic continuation of the integral
arising from the continuous spectrum in \eqref{waverep_K}, we substitute  $r\mapsto 1/2+ir$
and we observe that as a function of $s\in\mathbb{C}$ the integral converges locally absolutely and
uniformly for $\Re(s)>1-2n$ satisfying $\mathrm{Re}(s)\neq1/2-2\ell$ ($\ell=0,\ldots, n-1$).
To obtain the meromorphic continuation across the lines $\mathrm{Re}(s)=1/2-2\ell$,
we mimic the method used, e.g., in the proof of Theorem 2 of \cite{JKvP10} or in \cite{vP10}.
Namely, for $\ell=0$, we move the line of integration to $\mathrm{Re}(s)=1/2+\varepsilon$
for some $\varepsilon>0$ sufficiently small such that $\mathcal{E}^{\mathrm{par}}_{p_{k}}(z,s)$
has no poles in the strip $1/2-\varepsilon<\Re(s)<1/2+\varepsilon$. Using twice the residue theorem
and moving back the line of integration to $\mathrm{Re}(s)=1/2$, we obtain the meromorphic continuation
of the integral to the strip $-3/2<\Re(s)\leq1/2$. Repeating this process for $\ell=1,\ldots, n-1$, we
obtain the desired meromorphic continuation. Multiplying the integral by $\Gamma(s-1/2)^{-1}$,
the only poles arising in this process
are at  $s=1-\rho-2\ell$, where $\rho $ is a pole of the Eisenstein series\ $\mathcal{E}
_{p_{k}}^{\mathrm{par}}$ belonging to the line segment $\left( 1/2,1\right] $
and at $s=\rho -2\ell$, where $\rho $ is a pole of the Eisenstein series $
\mathcal{E}_{p_{k}}^{\mathrm{par}}$ such that $\mathrm{Re}(\rho)<1/2$ and $\ell=0,\ldots,n-1$.

\vskip .10in
Since $n\in\mathbb{N}$ was chosen arbitrarily, this proves the meromorphic continuation of $K_{s}(z,w)$
to the whole $s$-plane.
\end{proof}

\vskip .15in
\section{The wave representation of hyperbolic Eisenstein series}
\vskip .10in
We now express the hyperbolic Eisenstein series recalled in subsection \ref{2.2}
in terms of the wave distribution. To do this, we first establish an integral representation
for hyperbolic Eisenstein series using the automorphic kernel $K_{s}(z,w)$ and then we apply Theorem \ref{PoencViaWave}.

\vskip .10in
\begin{proposition}\label{HypViaPoencTh}
Let $\gamma \in \Gamma $ be a primitive, hyperbolic element of $\Gamma$ and let
$\mathcal{E}_{\gamma }^{\mathrm{hyp}}(z,s)$ be the hyperbolic Eisenstein series
associated to the closed geodesic $L_{\gamma}$ on $M$ in the homotopy class
determined by $\gamma$.
Then, for $z, w \in M$ and $s\in\mathbb{C}$ with $\Re(s)>1$, we have the identity
\begin{align}
\label{IntrepHypeis}
\int\limits_{L_{\gamma}}K_{s}(z,w)ds_{\mathrm{hyp}}(w)
=\frac{2^{s-1}\Gamma( \frac{s}{2})^{2}}{\Gamma (s)}
\,\mathcal{E}_{\gamma }^{\mathrm{hyp}}(z,s)\,.
\end{align}
\end{proposition}

\begin{proof}
Let $\gamma \in \Gamma $
be a primitive, hyperbolic element of $\Gamma$ with scaling-matrix $\sigma_{\gamma}\in\mathrm
{PSL}_{2}(\mathbb{R})$, so then
\begin{align*}
\sigma_{\gamma}^{-1}\gamma\sigma_{\gamma}=
\begin{pmatrix} e^{\ell_{\gamma}/2}&0\\0&e^{-\ell_{\gamma}/2}\end{pmatrix},
\end{align*}
where $\ell_{\gamma}$ denotes the hyperbolic length of $L_{\gamma}$. Without loss of generality,
we may assume that $\sigma_{\gamma}$ is the identity matrix, so then $L_{\gamma}=p(\mathcal{L}_{\gamma})$,
where $\mathcal{L}_{\gamma}=\mathcal{L}$ is the positive $y$-axis.
Let $\Gamma_{\gamma}:=\langle\gamma\rangle$ be the stabilizer subgroup of $\gamma$ in $\Gamma$;
$\Gamma_{\gamma}$ is isomorphic to $\mathbb{Z}$ with generator $\gamma$.
Let us rewrite the automorphic kernel $K_{s}(z,w)$ as
\begin{align}
K_{s}(z,w)
&=\sum\limits_{\eta \in \Gamma }\cosh (d_{\mathrm{hyp}}(\eta z, w))^{-s}
=\sum\limits_{\eta \in \Gamma_{\gamma}\backslash \Gamma }\,\sum\limits_{\eta' \in \Gamma_{\gamma}}
\cosh (d_{\mathrm{hyp}}(\eta'\eta z, w))^{-s}\notag\\
&=\sum\limits_{\eta \in \Gamma_{\gamma}\backslash \Gamma }\,\sum\limits_{n\in\mathbb{Z}}
\cosh (d_{\mathrm{hyp}}(\eta z, e^{-n \ell_{\gamma}}w))^{-s}.\label{Kseries}
\end{align}
Using elementary considerations involving counting functions, one has that the series in \eqref{Kseries}
converges absolutely and locally uniformly for  $z, w \in M$ and $s\in\mathbb{C}$ with $\Re(s)>1$;
see \cite{GJM08} for details.  Therefore, we get
\begin{equation}\label{intLgammaK1}
\int\limits_{L_{\gamma}}K_{s}(z,w)ds_{\mathrm{hyp}}(w)=
\sum\limits_{\eta \in \Gamma_{\gamma}\backslash \Gamma }\,\sum\limits_{n\in\mathbb{Z}}
\,\,\int\limits_{L_{\gamma}}\cosh (d_{\mathrm{hyp}}(\eta z, e^{n \ell_{\gamma}}w))^{-s}ds_{\mathrm{hyp}}(w).
\end{equation}
Set $z':=\eta z$ and use the
coordinates $\rho=\rho(w)$ and $\theta=\theta(w)$ from subsection \ref{2.1}.
If we analyze each term in \eqref{intLgammaK1}, we can write
\begin{align}
&\int\limits_{L_{\gamma}}\cosh (d_{\mathrm{hyp}}(z', e^{n \ell_{\gamma}}w))^{-s}ds_{\mathrm{hyp}}(w)
=\int\limits_{0}^{\ell_{\gamma}}\cosh (d_{\mathrm{hyp}}(z', e^{n \ell_{\gamma}+\rho}i))^{-s}d\rho=
\notag\\
&\int\limits_{n \ell_{\gamma}}^{(n+1) \ell_{\gamma}}
\cosh (d_{\mathrm{hyp}}(z', e^{\rho}i))^{-s}d\rho
=\cosh (d_{\mathrm{hyp}}(z', \mathcal{L}_{\gamma}))^{-s}\int\limits_{n \ell_{\gamma}}^{(n+1) \ell_{\gamma}}
\cosh (\rho-\log(\vert z'\vert ))^{-s}d\rho.\label{changecosh}
\end{align}
All equalities in \eqref{changecosh} are immediate except for the last one.  For this, we use
Theorem 7.11.1 from \cite{Be95} which is an identity for right-angled hyperbolic
triangles, namely
\begin{align*}
\cosh (d_{\mathrm{hyp}}(z', e^{\rho}i))&=\cosh (d_{\mathrm{hyp}}(\vert z'\vert i, e^{\rho}i))
\cosh (d_{\mathrm{hyp}}(z', \mathcal{L}))\\&=\cosh (\rho-\log(\vert z'\vert ))
\cosh (d_{\mathrm{hyp}}(z', \mathcal{L}_{\gamma})).
\end{align*}
Substituting \eqref{changecosh} into \eqref{intLgammaK1}, we arrive at the formula
\begin{align}
\int\limits_{L_{\gamma}}K_{s}(z,w)ds_{\mathrm{hyp}}(w)&=
\sum\limits_{\eta \in \Gamma_{\gamma}\backslash \Gamma }
\cosh (d_{\mathrm{hyp}}(\eta z, \mathcal{L}_{\gamma}))^{-s}\,
\sum\limits_{n\in\mathbb{Z}}\int\limits_{n \ell_{\gamma}}^{(n+1)
\ell_{\gamma}}
\cosh (\rho-\log(\vert \eta z\vert ))^{-s}d\rho\notag\\
&=\sum\limits_{\eta \in \Gamma_{\gamma}\backslash \Gamma }
\cosh (d_{\mathrm{hyp}}(\eta z, \mathcal{L}_{\gamma}))^{-s}\,
\int\limits_{-\infty }^{\infty}\cosh (\rho-\log(\vert \eta z\vert ))^{-s}d\rho.\label{intLgammaK2}
\end{align}
Using \eqref{f1} with $\nu=s$ and $r=0$, we get that
\begin{align}
\label{intcosh}
\int\limits_{-\infty }^{\infty}
\cosh (\rho-\log(\vert \eta z\vert ))^{-s}d\rho=2 \,\int\limits_{0}^{\infty}
\cosh (\rho)^{-s}d\rho
=\frac{2^{s-1}\,\Gamma( \frac{s}{2})^{2}}{\Gamma (s)}.
\end{align}
Finally, substituting \eqref{intcosh} into \eqref{intLgammaK2}, we arrive at the formula
\begin{align*}
\int\limits_{L_{\gamma}}K_{s}(z,w)ds_{\mathrm{hyp}}(w)&=
\frac{2^{s-1}\,\Gamma( \frac{s}{2})^{2}}{\Gamma (s)}
\sum\limits_{\eta \in \Gamma_{\gamma}\backslash \Gamma } \cosh (d_{\mathrm{hyp}}(\eta z, \mathcal{L}_{\gamma}))^{-s}=
\frac{2^{s-1}\,\Gamma( \frac{s}{2})^{2}}{\Gamma (s)}\,
\mathcal{E}_{\gamma }^{\mathrm{hyp}}(z,s),
\end{align*}
which proves the assertion.
\end{proof}

A direct consequence of Proposition \ref{HypViaPoencTh} is that
we can express the hyperbolic Eisenstein series $\mathcal{E}_{\gamma }^{\mathrm{hyp}}(z,s)$
as an integral over $L_{\gamma }$ of the wave distribution applied to $g_{s}$, up to a multiplicative
factor which is a universal function of $s$.

\vskip .10in
\begin{theorem}
Let $\gamma \in \Gamma $ be a primitive, hyperbolic element of $\Gamma$ and let
$\mathcal{E}_{\gamma }^{\mathrm{hyp}}(z,s)$ be the hyperbolic Eisenstein series
associated to the closed geodesic $L_{\gamma}$ on $M$ in the homotopy class
determined by $\gamma$.
Then, for $z, w \in M$ and $s\in\mathbb{C}$ with $\Re(s)>1$, we have the representation
\begin{equation*}
\mathcal{E}_{\gamma }^{\mathrm{hyp}}(z,s)=
\frac{2^{1-s}\,\Gamma (s)}{\Gamma( \frac{s}{2})^{2}}\,
\int\limits_{L_{\gamma }}
\mathcal{W}_{M,\frac{1}{4}}(z,w)(g_{s})ds_{\mathrm{hyp}}(w),
\end{equation*}
which converges absolutely and uniformly on compacta; here, the test function $g_{s}$ is given by
\eqref{g_s}.
\end{theorem}

\begin{proof}
The assertion follows immediately by combining Proposition \ref{HypViaPoencTh}
with Theorem \ref{PoencViaWave}.
\end{proof}

\vskip .10in
\begin{remark}\rm
The spectral expansion for hyperbolic Eisenstein series was established
in \cite{JKvP10}. By comparing this spectral expansion to the spectral
expansion \eqref{SpectralExpP} of $K_s(z,w)$, one gets
an alternative proof of Proposition \ref{HypViaPoencTh}.
Conversely, substituting the spectral expansion \eqref{SpectralExpP}
into \eqref{IntrepHypeis}, Proposition \ref{HypViaPoencTh}
gives rise to a new proof of the spectral expansion of hyperbolic Eisenstein series.
\end{remark}

\vskip .10in
\begin{remark}\rm
As proven in \cite{JKvP10}, for any $z\in M$, the hyperbolic Eisenstein
series $\mathcal{E}_{\gamma }^{\mathrm{hyp}}(z,s)$
admits a meromorphic continuation to the whole complex $s$-plane.
Alternatively, this can be proved along the lines of the proof of Theorem \ref{Poen_cont}
by starting with the identity
\begin{align*}
&\frac{2^{s+2n-1}\,\Gamma( \frac{s}{2})^{2}}{\Gamma(s+2n)}\,
\,\mathcal{E}_{\gamma }^{\mathrm{hyp}}(z,s)
=\sum\limits_{\lambda _{j} \geq 0}\frac{H(t_j,g_{s+2n})}{h_{n}(t_j,s)} \,\psi _{j}(z)
\int\limits_{L_{\gamma}}\psi _{j}(w) ds_{\mathrm{hyp}}(w)
\notag\\
&
+\frac{1}{4\pi }\sum\limits_{k=1}^{p_{\Gamma }}\,\int\limits_{-\infty}^{\infty}
\frac{H(r,g_{s+2n})}{h_{n}(r,s)}\,
\mathcal{E}_{p_{k}}^{\mathrm{par}}(z,1/2+ir)
\int\limits_{L_{\gamma}}\overline{\mathcal{E}_{p_{k} }^{\mathrm{par}}}(w,1/2+ir)ds_{\mathrm{hyp}}(w)\,
dr,
\end{align*}
which is deduced for $s\in\mathbb{C}$ with $\mathrm{Re}(s)>1$
by combining Proposition \ref{HypViaPoencTh} with relation \eqref{waverep_K}.
\end{remark}

\vskip .15in
\section{The wave representation of elliptic Eisenstein series}\label{6}
\vskip .10in
We now express the elliptic Eisenstein series recalled in subsection \ref{2.2}
in terms of the wave distribution. We start by recalling the following relation
of the elliptic Eisenstein series to the automorphic  kernel $K_{s}(z,w)$,
which was first given in \cite{vP10}.

\vskip .10in
\begin{lemma}\label{lemma_rel_ell_aut}
For $z, w \in M$ with $z\not = w$, and $s\in\mathbb{C}$ with $\Re(s)>1$,
the elliptic Eisenstein series $\mathcal{E}_{w}^{\mathrm{ell}}(z,s)$ associated
to the point $w$ satisfies the relation
\begin{align}
\mathcal{E}_{w}^{\mathrm{ell}}(z,s)=\frac{1}{\mathrm{ord}(w)}
\sum\limits_{k=0}^{\infty}
\frac{\left( \frac{s}{2}\right) _{k}}{k!}K_{s+2k}(z,w)\,.
\label{EllViaPoencare}
\end{align}
\end{lemma}

\begin{proof}
The proof of the assertion is immediate applying the general binomial theorem
and is given in Lemma 3.3.8 of \cite{vP10}.
\end{proof}

For technical reasons, see Remark \ref{problem_k_ell},
we will have to consider the following
two test functions.

\vskip .10in
\begin{definition}
For $s\in \mathbb{C}$ with $\Re(s)>1$, $\beta \in
\mathbb{C}$ with $\Re(\beta )>0$, and $u\in\mathbb{R}^{+}$, we set
\begin{align}\label{def_gsbeta}
g_{s,\beta}(u)&:=g_{s}(u)\tanh(u)^{\beta }.
\end{align}
For $s,\beta \in \mathbb{C}$ with $\mathrm{Re}(\beta)>\mathrm{Re}(s-1)>0$, and $u\in\mathbb{R}^{+}$, we set
\begin{align}\label{def_Gsbeta}
G_{s,\beta}(u)&:=g_{s}(u)\tanh(u)^{\beta+1-s } \,
F\Bigl(\frac{1}{4},\frac{3}{4};\frac{s}{2}+\frac{1}{2};\frac{1}{\cosh(u)^2}\Bigr).
\end{align}
Here, $g_s(u)$ is given by \eqref{g_s} and $F(a,b;c;z)$ denotes the Gauss hypergeometric function.
\end{definition}
Since $\mathrm{Re}(1/4+3/4-s/2-1/2)<0$ for $\mathrm{Re}(s)>1$, the
hypergeometric series in \eqref{def_Gsbeta} converges uniformly
for all $u \in \mathbb{R}$. Therefore, $G_{s,\beta }(u)$ is
well-defined on $\mathbb{R}^{+}$ including $u=0$ for the given
range of complex parameters $s$ and $\beta$.

\vskip .10in
\begin{lemma}\label{intformula}
Let $s\in\mathbb{C}$ with $\Re(s)>1$.
\begin{enumerate}
\item[(i)]
For $\beta\in\mathbb{C}$ with $\Re(\beta)>8$, we have the bound
\begin{equation} \label{H_r,k bound}
H(r,g_{s+2k,\beta})=O(k^{-( \Re(\beta)-8)/2} \, \vert r\vert^{-3})  \, \textrm{ as  }r \rightarrow \pm\infty, k \rightarrow \infty.
\end{equation}
\item[(ii)]
For $\beta\in\mathbb{C}$ with $\Re(\beta)>0$, we have the bound
\begin{equation} \label{H_r,k bound zero}
H(r,g_{s+2k,\beta})=O(k^{-\Re(\beta)/2})\, \textrm{ as  } k \rightarrow \infty,
\end{equation}
uniformly in $r$, for $r\in [ -1,1 ]$.
\end{enumerate}
The implied constants in (i) and (ii) depend uniformly on $s$ and $\beta$.
\end{lemma}

\begin{proof}
\emph{Part (i):}
By definition, we have
\begin{align}
H(r,g_{s+2k,\beta})=\frac{2^{3/2}\,\sqrt{\pi}\,\Gamma(\alpha_k)}{\Gamma(\alpha_k+\frac{1}{2})} \,\int\limits_{0}^{\infty}\cos (ur)
\tanh(u)^{\beta }  \cosh(u)^{-\alpha_k}du,
\label{def_altH}
\end{align}
where we have set $\alpha_k:=s-1/2+2k$. Integrating by parts three times, we find
$$
H(r,g_{s+2k,\beta})=\frac{2^{3/2}\,\sqrt{\pi}\,\Gamma(\alpha_k)}
{\Gamma(\alpha_k+\frac{1}{2})}\frac{1}{r^{3}} \,
\int\limits_{0}^{\infty}\sin(ur)\tanh(u)^{\beta-3}\cosh(u)^{-(\alpha_k+6)}S(\alpha_k, \beta;u)du
$$
with
$$
S(\alpha_k, \beta;u):=\sum _{j=0}^{3}P_{j}(\alpha_k, \beta) \sinh(u)^{2j},
$$
where $P_{j}(\alpha_k, \beta)$ is a polynomial in two variables $\alpha_k$ and $\beta$ of
a degree equal to three; here, we used the fact that $g_{s+2k,\beta} \in S^{\prime }(\mathbb{R}^{+},1/2+\eta)$
for any $\eta \in (0, \mathrm{Re}(s)-1)$ assuming $\Re(\beta)>2$. By the bound
$$
\vert S( \alpha_k, \beta;u) \vert \ll \vert \alpha_k \vert^{3} \sum _{j=0}^{3}\sinh(u)^{2j} \ll k^{3} \sum _{j=0}^{3}\sinh(u)^{2j},
$$
we get
\begin{align*}
&\left\vert\int\limits_{0}^{\infty}\sin(ur)\tanh(u)^{\beta-3}\cosh(u)^{-(\alpha_k+6)}S(\alpha_k, \beta;u)du\right\vert\\
&\ll   k^3\sum\limits_{j=0}^{3}\int\limits_{0}^{\infty} \sinh(u)^{\mu_j}  \cosh(u)^{-\nu}  du
=\frac{k^3}{2}\,\sum\limits_{j=0}^{3}\frac{\Gamma(\frac{\mu_j+1}{2})\Gamma( \frac{\nu-\mu_j}{2})}
{\Gamma(\frac{\nu+1}{2})}
\end{align*}
with $\mu_j:=\Re (\beta)-3+2j$ and $\nu:=\Re(\alpha_k)+\Re(\beta)+3$; here, the last equality
follows from formula 3.512.2 of \cite{GR07} provided that $\Re(\mu_j)>-1$ and $\Re(\mu_j-\nu)<0$.

\vskip .10in
An application of Stirling's formula yields, for any fixed real $y$, the asymptotics
$\Gamma(x)/\Gamma(x+y)\sim  x^{-y}$, as $x\to\infty$. Hence, we have the bounds
\begin{equation*}
\frac{\Gamma(\alpha_k)}{\Gamma(\alpha_k +\frac{1}{2})}=O(k^{-1/2})\,\,\, \text{  and      }\,\,\,\,  \frac{\Gamma( \frac{\nu-\mu_j}{2})}
{\Gamma(\frac{\nu+1}{2})}=
\frac{\Gamma( \frac{\Re(\alpha_k)}{2}+3-j)}
{\Gamma(\frac{\Re (\beta)}{2}+\frac{\Re(\alpha_k)}{2}+2)}
=O(k^{-(\Re (\beta)/2) +1-j}),
\end{equation*}
as $k \rightarrow \infty.$ Hence, summing up, we find
\begin{equation*}
H(r,g_{s+2k,\beta})=O(k^{-( \Re(\beta)-7)/2} \, \vert r\vert^{-3})  \, \textrm{ as  }\vert r\vert,k \rightarrow \infty,
\end{equation*}
where the implied constant depends uniformly on $s$ and $\beta$.

\vskip .10in
\emph{Part (ii):} In the case when $r\in[-1,1]$, using the trivial bound $\vert \cos (ur) \vert \leq 1$, we get immediately from \eqref{def_altH} that
$$
\vert H(r,g_{s+2k,\beta}) \vert \ll \frac{1}{k^{1/2}} \int\limits_{0}^{\infty} \sinh(u)^{\Re (\beta)}  \cosh(u)^{-(\Re(\alpha_k)+\Re (\beta))}du\, \text{  as } k\rightarrow \infty,
$$
uniformly in $r$. The bound \eqref{H_r,k bound zero} is now obtained in the same way as above, using the representation
of this integral in terms of Gamma factors and then applying Stirling's formula.

This completes the proof.
\end{proof}

\vskip .10in
\begin{lemma} \label{lemma_>8}
For $s, \beta \in \mathbb{C}$ with $ \Re(s)>1$ and $\Re (\beta)> \Re(s)+8$,
we have the relation
\begin{align}\label{relHgHG}
H(r,G_{s, \beta})=
 \sum\limits_{k=0}^{\infty} \frac{\left(\frac{s}{2}\right)_{k}}{k!} H(r, g_{s+2k, \beta}).
\end{align}
\end{lemma}

\begin{proof}
Let $s, \beta \in \mathbb{C}$ with $ \Re (\beta)>\Re(s)+8>9$ and
$u\in\mathbb{R}^{+}$. We first prove the relation
\begin{align}\label{Gseries2}
G_{s,\beta }(u)=\sum_{k=0}^{\infty }
\frac{\left(\frac{s}{2}\right)_{k}}
{k! }g_{s+2k,\beta}(u).
\end{align}
The equality \eqref{Gseries2} obviously holds true for $u=0$, since both sides of \eqref{Gseries2}
are well-defined and equal to zero. Now, assume $u>0$.
Using formula $F(a,b;c;z)=(1-z)^{c-a-b} F(c-a,c-b;c;z)$ (see \cite{GR07}, formula 9.131.1),
we deduce
\begin{align}\label{Gseries3}
G_{s,\beta }(u)=g_{s}(u)\,\tanh(u)^{\beta}\,
F\Bigl(\frac{s}{2}-\frac{1}{4},\frac{s}{2}+ \frac{1}{4}; \frac{s}{2}+\frac{1}{2}; \frac{1}{\cosh(u)^2}\Bigr).
\end{align}
Then, substituting the equality
$$
g_{s,\beta}(u)=\frac{2^s\,\Gamma(\frac{s}{2}-\frac{1}{4})\Gamma(\frac{s}{2}+\frac{1}{4})}{\Gamma(s)}\,\tanh(u)^{\beta }  \cosh(u)^{-(s-1/2)}
$$
and the well-known series expansion of the hypergeometric function
into \eqref{Gseries3}, we get
\begin{align*}
G_{s,\beta }(u)&= \tanh(u)^{\beta}  \frac{2^s}{\Gamma(s)}\,\sum_{k=0}^{\infty }
\frac{\Gamma(\frac{s}{2}-\frac{1}{4}+k)\Gamma(\frac{s}{2}+\frac{1}{4}+k)}
{k!\,(\frac{s}{2}+\frac{1}{2})_{k}}\cosh(u)^{-(s+2k-1/2)}\\
&=\sum_{k=0}^{\infty }\frac{2^{-2 k}\,(s)_{2k}}{k!\,(\frac{s}{2}+\frac{1}{2})_{k}}
g_{s+2k,\beta}(u).
\end{align*}
The relation \eqref{Gseries2} now follows from the dimidiation formula
$2^{-2 k}\,(2a)_{2k} (a+1/2)_{k}^{-1}=(a)_k$ for the Pochhammer symbol
with $a:=s/2$.
Finally, using the relation \eqref{Gseries2}, we get
\begin{align*}
H(r,G_{s, \beta})=2\int\limits _{0} ^{\infty} \left( \sum_{k=0}^{\infty} \frac{\left(\frac{s}{2}\right)_{k}}{k!}
g_{s+2k, \beta}(u) \right) \cos(ur) du= \sum_{k=0}^{\infty} \frac{\left(\frac{s}{2}\right)_{k}}{k!} H(r, g_{s+2k, \beta}).
\end{align*}
The interchange of the sum and the integral defining $H(r, g_{s+2k, \beta})$ is justified by the bound
$$
\Bigl\vert \frac{\left(\frac{s}{2}\right)_{k}}{k!} H(r, g_{s+2k, \beta})\Bigr\vert \ll
 k^{-(\Re(\beta)/2-\Re(s)/2-3)} \ll k^{-(1+\varepsilon)},
$$
for $\varepsilon>0$ with $\Re (\beta)=2\varepsilon+\Re (s)+8$, which can be deduced
from the bounds  \eqref{H_r,k bound} and \eqref{H_r,k bound zero}
together with
$$
\Bigl\vert\frac{\left(\frac{s}{2}\right)_{k}}{k!}\Bigr\vert \ll k^{(\Re (s)/2)-1}.
$$
This completes the proof of the relation \eqref{relHgHG}.
\end{proof}

\vskip .10in
\begin{lemma} \label{bound_with_beta}
Let $s,\beta \in\mathbb{C}$ with $\Re(\beta )>\Re (s+1)>2$.
\begin{enumerate}
\item[(i)]
For $z, w\in M$, the wave distribution $\mathcal{W}_{M,\frac{1}{4}}(z,w)(g_{s,\beta })$ is well-defined and
holomorphic as a function of $\beta$.
\item[(ii)] For $z, w\in M$ with $z \neq w$,
the wave distribution $\mathcal{W}_{M,\frac{1}{4}}(z,w)(g_{s,\beta })$ admits a holomorphic continuation
in $\beta$ to $\Re(\beta )>0$.
Further,
for $\Re(\beta )>0$, assuming $\Re(s)=\sigma \geq \sigma_{0} > 1$
and $z \neq w$,
there exists a constant $C > 1$ depending on $d_{\textrm{hyp}}(z,w)$ such that
\begin{equation*}
\mathcal{W}_{M,\frac{1}{4}}(z,w)(g_{s,\beta }) = O(C^{-(\sigma-\sigma_{0})})
\,\,\,\,\,\textrm{as $\sigma \rightarrow \infty$.}
\end{equation*}
\end{enumerate}
\end{lemma}

\begin{proof}
\emph{Part (i):}
For $s,\beta \in\mathbb{C}$ with $\Re(\beta )>\Re (s+1)>2$, a straight forward computation
shows that $g_{s,\beta }(u)\in S^{\prime }(\mathbb{R}^{+},1/2)$, and that the first three derivatives
of $g_{s,\beta }(u)$ are integrable and have a limit as $u\to\infty$.
Hence, by part (i) of Proposition 5, the wave distribution $\mathcal{W}_{M,\frac{1}{4}}(z,w)(g_{s,\beta })$
is well-defined and holomorphic as a function of $\beta$.

\vskip .10in
\emph{Part (ii):}
Moreover, for $\Re(\beta)>2$ and $\Re(s)>1$, the function $g_{s,\beta }(u)$ satisfies the
assumptions of Proposition \ref{uniquness}, part (iii).
Hence, for $\delta > 0$, whose choice will be clarified during the course of the proof, we
can use Proposition \ref{uniquness} to write
\begin{equation}\label{equW}
\mathcal{W}_{M,\frac{1}{4}}(z,w)(g_{s,\beta })=
\int\limits_{0}^{\delta }F(z,w;u)\,g_{s,\beta }^{(3)}(u)du+
\int\limits_{\delta }^{\infty}F(z,w;u)\,g_{s,\beta }^{(3)}(u)du.
\end{equation}
Since the right hand side of \eqref{equW} is meaningful for
$\Re(\beta)> 0$, due to bound (ii) and (i) from Proposition
\ref{uniquness}, this provides the holomorphic continuation of
$\mathcal{W}_{M,\frac{1}{4}}(z,w)(g_{s,\beta })$ in the variable
$\beta$ to the half-plane $\Re(\beta) >0$, as asserted.

\vskip .10in
Now, let $\Re(\beta)>0$ and assume $\Re(s)=\sigma \geq \sigma_{0} > 1$.
The bound (i) from Proposition \ref{uniquness} together
with the inequality $\tanh(u) \leq 1$ implies the bound
\begin{equation}\label{betabound1}
\left|
\int\limits_{\delta }^{\infty}F(z,w;u)\,g_{s,\beta }^{(3)}(u)du
\right|= O\left(C_{\delta}^{-\sigma+1}\right)
\end{equation}
for some constant $C_{\delta}>1$, as $\sigma \rightarrow \infty$.
Further, using the notation and results of
Proposition \ref{uniquness}, we can integrate by parts to write
\begin{equation*}
\int\limits_{0}^{\delta }F(z,w;u)\,g_{s,\beta }^{(3)}(u) du=
\int\limits_{0}^{\delta }\left(P_{M,\frac{1}{4}}(z,w;iu) + P_{M,\frac{1}{4}}(z,w;-iu)\right)
g_{s,\beta }(u)du +O\bigl(\widetilde{C}_{\delta}^{-\sigma+1}\bigr)
\end{equation*}
for some constant $\widetilde{C}_{\delta}>1$, as $\sigma \rightarrow \infty$; in particular, we
used identity \eqref{thirdderivativeFtildeequalsP} which holds by analytic continuation
for $\zeta=iu$ with $0<u\leq\delta$ and $\delta$ chosen as below. Since $z \neq w$, the function
\begin{equation*}
\widetilde{P}(z,w;u):= P_{M,\frac{1}{4}}(z,w;iu) + P_{M,\frac{1}{4}}(z,w;-iu)
\end{equation*}
is real valued and continuous as $u$ approaches zero.  Therefore, we can choose $\delta > 0$ sufficiently small
so that the function $\widetilde{P}(z,w;u)$ has one sign when $u \in (0,\delta)$, hence
\begin{equation*}
\left|\int\limits_{0}^{\delta }\widetilde{P}(z,w;u)g_{s,\beta }(u)du \right|
\leq \int\limits_{0}^{\delta }\left|\widetilde{P}(z,w;u)\right| g_{\sigma}(u)du=
\left \vert \int\limits_{0}^{\delta }\widetilde{P}(z,w;u) g_{\sigma}(u)du \right \vert,
\end{equation*}
where the inequality follows, since $\tanh(u) \leq 1$, and the last equality follows, since the function
$\widetilde{P}(z,w;u)$ has one sign and $g_{\sigma}(u)$ is positive. Integrating by parts once again, we get
\begin{align}
 \int\limits_{0}^{\delta }\widetilde{P}(z,w;u) g_{\sigma}(u)du&=
 \int\limits_{0}^{\delta }F(z,w;u)g_{\sigma}^{(3)}(u)du
 =
 \mathcal{W}_{M,\frac{1}{4}}(z,w)(g_{\sigma})-\int\limits_{\delta}^{\infty }
 F(z,w;u)g_{\sigma}^{(3)}(u)du
 \notag \\
 &=
 \mathcal{W}_{M,\frac{1}{4}}(z,w)(g_{\sigma})+ O\left(C_{\delta}^{-\sigma+1}\right)
\label{betabound2}
\end{align}
as $\sigma \rightarrow \infty$, by computations similar to the ones used
to obtain the bound \eqref{betabound1}. Now, Theorem \ref{PoencViaWave}
yields $\mathcal{W}_{M,\frac{1}{4}}(z,w)(g_{\sigma})=K_{\sigma}(z,w)$, hence,
from the series definition \eqref{auto_kernel_series}, we
derive the bound
\begin{align*}
\vert\mathcal{W}_{M,\frac{1}{4}}(z,w)(g_{\sigma})\vert
 &\leq \widetilde{C}^{-(\sigma-\sigma_{0})} K_{\sigma_0}(z,w) \ll \widetilde{C}^{-(\sigma-\sigma_{0})}
\end{align*}
for a constant $\widetilde{C}>1$ such that  $\widetilde{C}<\cosh(d_{\textrm{hyp}}(z,\eta w))$
for any $\eta\in\Gamma$.
The proof is complete by substituting the last bound together with the bounds \eqref{betabound1}
and \eqref{betabound2} into equality \eqref{equW} and letting
$C= \min \{C_{\delta},\widetilde{C}_{\delta},\widetilde{C}\} > 1$.
\end{proof}

\vskip .10in
\begin{theorem}\label{EllEisViaWaveTh}
Let $s,\beta \in \mathbb{C}$ with $\Re(\beta )> \Re (s+1)> 2$.
\begin{enumerate}
\item[(i)]
For $z, w\in M$, the wave distribution $\mathcal{W}_{M,\frac{1}{4}}(z,w)(G_{s,\beta })$ is well-defined
and holomorphic as a function of $\beta$.
\item[(ii)] For $z, w\in M$ with $z \neq w$,
the wave distribution $\mathcal{W}_{M,\frac{1}{4}}(z,w)(G_{s,\beta })$ admits a holomorphic continuation
in $\beta$ to $\{\beta \in
\mathbb{C}\,|\,\Re(\beta)>0\} \cup \{0\}$, and we have the representation
\begin{align*}
\mathrm{ord}(w)\mathcal{E}_{w}^{\mathrm{ell}}(z,s)=\mathcal{W}_{M,\frac{1}{4}}(z,w)(G_{s,\beta })\Big\vert _{\beta =0}.
\end{align*}
\end{enumerate}
Here, the test function $G_{s,\beta}$ is given by \eqref{def_Gsbeta}.
\end{theorem}

\begin{proof}
\emph{Part (i):}
Let $s, \beta \in \mathbb{C}$ with $ \Re (\beta)>\Re(s+1)>2$. A straight forward computation
shows that, for any $\eta \in (0, \mathrm{Re}(s)-1)$, we have $G_{s,\beta }\in S^{\prime }(\mathbb{R}^{+},1/2+\eta)$
and $G_{s,\beta }^{(j)}(u)\exp(u(1/2+ \eta))$ is bounded by some integrable function on $\mathbb{R}^{+}$ for $j=1,2,3$.
Hence, by Proposition \ref{Wavedistr}, $\mathcal{W}_{M, \frac{1}{4}}(z,w)(G_{s,\beta })$ is well-defined and a holomorphic function
in $\beta$.

\vskip .10in
\emph{Part (ii):}
Let $s\in \mathbb{C}$  with $\Re(s)\geq \sigma_{0}>1$. Then, for $z, w\in M$ with $z \neq w$,
the wave distribution $\mathcal{W}_{M,\frac{1}{4}}(z,w)(g_{s+2k,\beta })$
is holomorphic in $\beta$ for any $k\in\mathbb{N}$. Further, by part (ii) of Lemma
\ref{bound_with_beta}, we get the bound
$$
\sum\limits_{k=0}^{\infty}\left|
\frac{\left(\frac{s}{2}\right)_{k}}{k!}\,\mathcal{W}_{M, \frac{1}{4}}(z,w)(g_{s+2k,\beta})\right|
\ll \sum\limits_{k=0}^{\infty} \frac{\bigl(\frac{\Re(s)}{2}\bigr)_{k}}{k!}\,   C^{-2k} <\infty,
$$
uniformly for any $\beta \in\mathbb{C}$ with $\Re(\beta)> 0$.
Hence, the series
$$
\sum\limits_{k=0}^{\infty}\frac{\left(\frac{s}{2}\right)_{k}}{k!}\,
\mathcal{W}_{M, \frac{1}{4}}(z,w)(g_{s+2k,\beta})
$$
converges absolutely and uniformly, and hence represents a
holomorphic function for $\Re(\beta)> 0$.

\vskip .10in
Further, for $\Re(\beta)$ large enough (e.g.~$\Re(\beta)> \Re(s)+8>9$), by the bounds \eqref{H_r,k bound} and \eqref{H_r,k bound zero} (both bounds are needed in the case of the integral), we are allowed to interchange the
sum over $k$ with the sum, resp.~the integral appearing in the definition of the wave distribution $\mathcal{W}_{M, \frac{1}{4}}(z,w)(g_{s+2k,\beta})$, and we obtain
\begin{align*}
&\sum\limits_{k=0}^{\infty}\frac{\left(\frac{s}{2}\right)_{k}}{k!}
\mathcal{W}_{M, \frac{1}{4}}(z,w)(g_{s+2k,\beta})=\\
&\sum_{\lambda_{j} \geq 0}H(t_{j}, G_{s, \beta})
\,\psi_{j}(z)\psi_{j}(w)
+ \frac{1}{4\pi }\sum\limits_{j=1}^{p_{\Gamma }}\,\int\limits_{-\infty}^{\infty}
H(r, G_{s, \beta})
\,\mathcal{E}_{p_{j}}^{\mathrm{par}}(z,1/2+ir)
\overline{\mathcal{E}_{p_{j} }^{\mathrm{par}}}(w,1/2+ir)dr,
\end{align*}
employing the relation \eqref{relHgHG}. Thus, for $\Re(\beta)\gg 0$, we have
\begin{align}\label{equ_Gg}
\sum_{k=0}^{\infty }
\frac{\left(\frac{s}{2}\right)_{k}}
{k! }\mathcal{W}_{M, \frac{1}{4}}(z,w)(g_{s+2k,\beta})
=\mathcal{W}_{M, \frac{1}{4}}(z,w)(G_{s,\beta }).
\end{align}
Since the left hand side of \eqref{equ_Gg} is holomorphic for $\Re(\beta)> 0$, this identity
provides the holomorphic continuation of $\mathcal{W}_{M, \frac{1}{4}}(z,w)(G_{s,\beta })$
to the half-plane $\Re(\beta)>0$. Moreover, letting $\beta=0$ on the left hand side of
\eqref{equ_Gg} and combining relation \eqref{EllViaPoencare} with Theorem \ref{PoencViaWave},
we deduce
$$
\sum_{k=0}^{\infty }
\frac{\left(\frac{s}{2}\right)_{k}}
{k! }\mathcal{W}_{M, \frac{1}{4}}(z,w)(g_{s+2k})=\textrm{\rm ord}(w)\mathcal{E}_{w}^{\mathrm{ell}}(z,s).
$$
Hence, by the principle of analytic continuation, we obtain the identity
$$
\textrm{\rm ord}(w)\mathcal{E}_{w}^{\mathrm{ell}}(z,s)=
\mathcal{W}_{M, \frac{1}{4}}(z,w)(G_{s,\beta })\Big|_{\beta = 0},
$$
thereby completing the proof of the theorem.
\end{proof}

\vskip .10in
\begin{remark}\label{problem_k_ell}\rm
The elliptic Eisenstein series can be viewed as the automorphic kernel associated
to the point-pair invariant
\begin{equation} \label{sinhpp}
k^{\mathrm{ell}}_{s}(z,w):=\sinh(d_{\mathrm{hyp}}(z,w))^{-s}.
\end{equation}
Having in mind relation \eqref{WaveAsAutKernel},
one may ask whether there is a more direct approach to the wave
representation of the elliptic Eisenstein series by computing the
Selberg/Harish-Chandra transform of \eqref{sinhpp} (see subsection \ref{2.3}).
To do this, we first write
$k^{\mathrm{ell}}_{s}(u(z,w))=k^{\mathrm{ell}}_{s}(u)=2^{-s}u^{-s/2}(u+1)^{-s/2}$,
where $u=u(z,w)$ is defined by \eqref{def_u}, and get
\begin{align*}
Q^{\mathrm{ell}}_{s}(v):&=\int\limits_{v}^{\infty}\frac{k^{\mathrm{ell}}_{s}(u)}{\sqrt{u-v}}du
= 2^{-s}\int\limits_{0}^{\infty}u^{-\frac{1}{2}}(u+v)^{-s/2}(u+v+1)^{-s/2}du
\\
&=\frac{2^{-s}\sqrt{\pi}\,\Gamma (s-\frac{1}{2})}{\Gamma (s)}\frac{(v+1)^{1/2-s/2}}{v^{s/2}}
\,F\Bigl(\frac{1}{2}, \frac{s}{2}; s; -\frac{1}{v}\Bigr);
\end{align*}
here, for the last equality we used formula 3.197.1.~of \cite{GR07} with
$\beta=v$, $\gamma=v+1$, $\nu =1/2$, $\mu = \rho =s/2$ keeping in mind that $\mathrm{Re}(s/2)>\mathrm{Re}(1/2-s/2)$.
Finally, we get
\begin{align*}
g^{\mathrm{ell}}_{s}(u):&=2\,Q^{\mathrm{ell}}_{s}\bigl(\sinh\bigl(\frac{u}{2}\bigr)^2\bigr)=
\frac{2^{-s}\sqrt{\pi}\,\Gamma (s-\frac{1}{2})}{\Gamma (s)}\frac{\cosh(\frac{u}{2})^{1-s}}{\sinh(\frac{u}{2})^{s}}
\,F\Bigl(\frac{1}{2}, \frac{s}{2}; s; -\frac{1}{\sinh(\frac{u}{2})^2}\Bigr)\\
&=
g_s(u)
\tanh(u)^{1-s}\,
F\Bigl(\frac{1}{4},\frac{3}{4};\frac{s}{2}+\frac{1}{2};\frac{1}{\cosh(u)^2}\Bigr),
\end{align*}
where the last formula follows from formula 9.134.1.~of
\cite{GR07} with $\alpha=1/2$, $\beta=s/2$ and $z=-1/(\sinh^{2}
(u/2))$. Thus, for $u\in\mathbb{R}^{+}$ with $u>0$, we have
$g^{\mathrm{ell}}_{s}(u)=G_{s,0} (u)$. However, the functions $G_{s,0} (u)$
and $g^{\mathrm{ell}}_{s}(u)$ are not defined for $u=0$ for $\Re(s)>1$.
Furthermore, due to the growth of $\tanh ^{1-s} (u)$ as
 $u\downarrow 0$, the Selberg/Harish-Chandra transform of \eqref{sinhpp}
 does not exist
 for $\mathrm{Re}(s) >2$. Therefore, it is not
possible to provide the above mentioned direct approach.
\end{remark}

\vskip .10in
\begin{remark}\rm
As proved in \cite{vP10}, for any $z, w \in M$ with $z\not= w$, the elliptic Eisenstein
series $\mathcal{E}_{w}^{\mathrm{ell}}(z,s)$
admits a meromorphic continuation to the whole complex $s$-plane.
The main step of this proof consist of using relation \eqref{EllViaPoencare} and
writing, now in terms of the wave representation,
\begin{align}
 \textrm{\rm ord}(w)\mathcal{E}_{w}^{\text{ell}}(z,s)
 &=
\sum\limits_{k=0}^{n }
\frac{\left(\frac{s}{2}\right)_{k}}
{k! }\mathcal{W}_{M, \frac{1}{4}}(z,w)(g_{s+2k})+
\sum\limits_{k=n+1}^{\infty}
\frac{\left(\frac{s}{2}\right)_{k}}
{k! }\mathcal{W}_{M, \frac{1}{4}}(z,w)(g_{s+2k})
\label{eq_split_elleis}
\end{align}
for any positive integer $n$. The second series in \eqref{eq_split_elleis}
represents a holomorphic function for $\Re(s) > 1-2n$;
the first sum in \eqref{eq_split_elleis} admits a meromorphic continuation
to the half-plane $\Re(s) > 1-2n$
as a consequence of Theorem \ref{Poen_cont}.
\end{remark}

\vskip .15in
\section{The wave representation of parabolic Eisenstein series }

\vskip .10in
We now develop a representation of the parabolic Eisenstein series \eqref{def_eis_par}
in terms of the wave distribution. We start by recalling that the parabolic Eisenstein
arises in the zeroth coefficient of the Fourier expansion of the following automorphic kernel
(see, e.g., \cite{He83}, \cite{Iwa02}).

\vskip .10in
\begin{definition}\rm
The hyperbolic Green's function $G_s(z,w)$ is defined for $z, w\in M$ with $z\not=w$,
and $s\in\mathbb{C}$ with $\Re(s)>1$, by the following series
\begin{align*}
G_s(z,w):=
\frac{1}{2\pi}\,\sum\limits_{\eta \in\Gamma}Q_{s-1}(1+2u(z,\eta w))\,,
\end{align*}
where $Q_{\nu}(\cdot)$ denotes the Legendre function of the second kind
and  $u(z,w)$ given by \eqref{def_u}.
\end{definition}
We note that $G_s(z,w)$ as a function of $z$ (with $w$ fixed), or as a function of $w$ (with $z$ fixed),
is not a $L^2$-function on $M$.

\vskip .10in
\begin{proposition}\label{prop_rel_par_greens}
Let $\mathcal{E}^{\mathrm{par}}_{p_{j}}(z,s)$ be the parabolic Eisenstein series
associated to the cusp $p_{j}$ with scaling matrix $\sigma_{p_{j}}\in\PSL_{2}(\mathbb{R})$ $(j=1,\dots, p_{\Gamma})$.
Furthermore, let $L_{p_j}:=\sigma_{p_{j}}L_{a}$,
where $L_a$ is the horocycle given by $L_a:=\{z\in\mathbb{H}\,|\, \mathrm{Re}(z)\in[0,1], \mathrm{Im}(z)=a\}$.
Then, for $z\in M$, $s\in\mathbb{C}$ with $\Re(s)>1$, and
$a\in \mathbb{R}$ with $a> \mathrm{Im}(\eta z)$ for any $\eta \in \Gamma$,
we have the identity
\begin{align*}
\mathcal{E}^{\mathrm{par}}_{p_{j}}(z,s)=(2 s-1)\,\mathrm{Im}(w)^{s-1}
\int\limits_{L_{p_j}}  G_s(z,w) ds_{\mathrm{hyp}}(w)\,.
\end{align*}
\end{proposition}

\begin{proof}
From the Fourier expansion of the hyperbolic Green's function with respect to
the variable $w$ (see, e.g., \cite{Iwa02}, Theorem 5.3.),
we deduce under the above assumptions
\begin{align*}
\int\limits_{0}^{1}  G_s\bigl(z,\sigma_{p_{j}}(x'+iy')\bigr) dx'
=(2s-1)^{-1}\,(y')^{1-s} \,\mathcal{E}^{\mathrm{par}}_{p_{j}}(z,s),
\end{align*}
which after the change of coordinates $w=x'+iy'\mapsto \sigma_{p_{j}}w$
yields the assertion.
\end{proof}

From Proposition \ref{prop_rel_par_greens}, it suffices to represent the 
automorphic kernel $G_s(z,w)$ in terms of the wave distribution.  
To do so, we first derive the following general statement.

\vskip .10in
\begin{proposition}\label{KernelviaCoshThm_new}
Let $k(z,w)$ be a point-pair invariant function on $M\times M$
and let us write $k(z,w)=k(u(z,w))=k(u)$ as a function of $u$ with $u(z,w)$
given by \eqref{def_u}. Suppose that the Selberg/Harish-Chandra
transform $h(r)$ of $k(u)$ exists and satisfies conditions (S1) and
(S2) of subsection \ref{2.3} together with the bound
$
h(r)=O\left( (1+\vert r\vert)^{-2}\right)
$
as $\mathrm{Im}(r)\rightarrow\pm\infty$ in the domain of condition (S2).
Furthermore, assume that the automorphic kernel $K(z,w)$ associated to
$k(u)$ can be realized as
\begin{align}\label{kernelViaM-F}
K(z,w)=\frac{1}{4\pi }\sum\limits_{\eta\in\Gamma}
\int\limits_{-\infty}^{\infty}r\tanh (\pi r)h(r)
P_{-\frac{1}{2}+ir}(\cosh (d_{\mathrm{hyp}}(z,\eta w)))dr,
\end{align}
where $P_{-\frac{1}{2}+ir}(\cdot)$ denotes the Legendre function of the first kind.
Then, for some $0<\delta <1/2$, and $z, w \in M$ with $z\not=w$, we have the
representation
\begin{align*}
K(z,w)=\frac{1}{2\pi^{3/2}\,i}\int\limits_{\mathrm{Re}(\nu)=1+\delta}^{}
\frac{2^{-\nu}\, \Gamma(\nu) }{\Gamma (\nu-\frac{1}{2})}
h\Bigl(i\Bigl(\frac{1}{2}-\nu\Bigr)\Bigr) \mathcal{K}_{\nu}(z,w)d\nu,
\end{align*}
where we have set
\begin{equation*}
\mathcal{K}_{\nu}(z,w):=\sum\limits_{k=0}^{\infty} \frac{b_k(\nu)}{k!}K_{\nu+2k}(z,w)
\end{equation*}
with $b_k(\nu):=(\frac{\nu}{2})_{k} (\frac{\nu+1}{2})_{k}/(\nu+\frac{1}{2})_{k}$.
\end{proposition}

\begin{proof}
To begin, we note that
$$
P_{-\frac{1}{2}+ir} (\cosh (d_{\mathrm{hyp}}(z,\eta w)))=P_{-\frac{1}{2}+ir}(1+2u(z,\eta w)).
$$
Set $u:=u(z,\eta w)>0$.
Using formula 8.820.4 from \cite{GR07}, we can represent the Legendre function as
$$
P_{-\frac{1}{2}+ir}(1+2u)= (1+u)^{-\frac{1}{2}+ir}
F\left(\frac{1}{2}-ir,\frac{1}{2}-ir;1;\frac{u}{u+1} \right).
$$
The application of formula 9.134.2 from \cite{GR07} then yields
$$
P_{-\frac{1}{2}+ir}(1+2u)=
(1+2u)^{-\frac{1}{2}+ir}
F\left(\frac{1}{4}-\frac{ir}{2},\frac{3}{4}-\frac{ir}{2};1;\frac{(2u+1)^2-1}{(2u+1)^2} \right).
$$
Formula 9.131.2 from \cite{GR07}, together with the duplication formula for the $\Gamma$-function,
then implies that the function $P_{-\frac{1}{2}+ir}(1+2u)$ can be represented as the sum of two
hypergeometric functions, namely, we have that
\begin{align}\label{eq_splitH}
P_{-\frac{1}{2}+ir}(1+2u)=H\Bigl(\frac{1}{2}+ir,u\Bigr)+H\Bigl(\frac{1}{2}-ir,u\Bigr)
\end{align}
with
\begin{align*}
H(\nu,u):= \frac{2^{\nu} \,\Gamma(1-2\nu)}{\Gamma(1-\nu)^{2} }\,(1+ 2u)^{-\nu}
F\Bigl(\frac{\nu}{2},\frac{\nu+1}{2};\nu+\frac{1}{2};\frac{1}{(1+ 2u)^{2}}\Bigr)\,.
\end{align*}
Hence, substituting \eqref{eq_splitH} with $u=u(z,\eta w)$ into \eqref{kernelViaM-F} and observing
that the function $r\tanh (\pi r)h(r)$ is even in $r$, we derive the representation
\begin{align}\label{IntegralRepp}
K(z,w)&=\frac{1}{2\pi }\sum\limits_{\eta\in\Gamma}
\int\limits_{-\infty}^{\infty}r\tanh (\pi r)h(r)
H\Bigl(\frac{1}{2}+ir,u(z,\eta w)\Bigr) dr .
\end{align}
By combining formulas 8.332.1 and 8.332.2 from \cite{GR07}, we obtain the identity
\begin{align*}
r\tanh (\pi r)=\frac{\Gamma(\frac{1}{2}+ir)\Gamma(\frac{1}{2}-ir)}{\Gamma(ir)\Gamma(-ir)}.
\end{align*}
Using this formula, together with the duplication formula for the $\Gamma$-function,
we obtain the equality
\begin{align*}
r\tanh (\pi r)
H\Bigl(\frac{1}{2}+ir,u\Bigr)=\frac{2^{-1/2-ir} \,\Gamma(\frac{1}{2}+ir)}{\sqrt{\pi}\,\Gamma(ir) }
(1+ 2u)^{-\frac{1}{2}-ir} F\Bigl(\frac{1}{4}+\frac{ir}{2},\frac{1}{2}+\frac{ir}{2};1+ir;\frac{1}{(1+ 2u)^{2}} \Bigr).
\end{align*}
Substituting $\nu:=1/2+ir$ in \eqref{IntegralRepp} therefore gives
\begin{align}\label{IntegralRep}
K(z,w)
&=
\sum\limits_{\eta\in\Gamma}\,
\int\limits_{\mathrm{Re}(\nu)=1/2} c(\nu)
\frac{ h\bigl(i\bigl(\frac{1}{2}-\nu\bigr)\bigr)}{(1+ 2u(z,\eta w))^{\nu}}
F\Bigl(\frac{\nu}{2},\frac{\nu}{2}+\frac{1}{2};\nu+\frac{1}{2};\frac{1}{(1+ 2u(z,\eta w))^{2}}\Bigr) d\nu\,
\end{align}
with
$$
c(\nu) := \frac{1}{2\pi^{3/2}i }\frac{2^{-\nu}\Gamma(\nu)}{\Gamma(\nu-1/2)}.
$$
It is immediate that the integrand in \eqref{IntegralRep} is a holomorphic
function for $1/2 \leq \mathrm{Re}(w) \leq 1+\delta<3/2$.
 The bound
on the function $h$ and the Stirling formula for the
$\Gamma$-function imply that we may apply Cauchy's formula and move the
line of integration in \eqref{IntegralRep} to the line
$\mathrm{Re}s=1+\delta$, thus obtaining that
\begin{equation} \label{Kfinal}
K(z,w)= \sum\limits_{\eta\in\Gamma}\,
\int\limits_{\mathrm{Re}(\nu)=1+\delta}^{}
c(\nu)
\frac{ h\bigl(i\bigl(\frac{1}{2}-\nu\bigr)\bigr)}{\cosh (d_{\mathrm{hyp}}(z,\eta w))^{\nu}}
F\Bigl(\frac{\nu}{2},\frac{\nu}{2}+\frac{1}{2};\nu+\frac{1}{2};\frac{1}{\cosh (d_{\mathrm{hyp}}(z,\eta w))^{2}}\Bigr) d\nu\,.
\end{equation}
Let $\widetilde{C}>1$ be a constant such that  $\widetilde{C}<\cosh(d_{\textrm{hyp}}(z,\eta w))$
for any $\eta\in\Gamma$, as in the proof of Lemma \ref{bound_with_beta}.
Since $\mathrm{Re}\left( \frac{\nu}{2}+\frac{\nu}{2}+\frac{1}{2}-\nu-\frac{1}{2}\right)=0$ and $\cosh(d_{\mathrm{hyp}} (z, \eta w))^{-2} <\widetilde{C}^{-2}<1$, the hypergeometric function in \eqref{Kfinal} converges uniformly in $\nu$ and is bounded as a function of $\nu$ and $\eta\in \Gamma$. Therefore, the Stirling formula for the gamma function and the bound for the function
$h$ imply that the integrand in \eqref{Kfinal} is uniformly bounded by
\begin{equation} \label{integrand bound}
\frac{1}{(1+\vert \nu \vert)^{3/2}}\cosh(d_{\mathrm{hyp}}(z,\eta w))^{-1-\delta}.
\end{equation}
\noindent
Therefore, the series in \eqref{Kfinal} is majorized by the automorphic kernel $K_{1+\delta}(z,w)$, which
allows us to interchange the sum and the integral in \eqref{Kfinal}.
\vskip .05in
Finally, arguing in the same manner, we deduce that
\begin{align*}
\sum_{\eta\in\Gamma}\cosh (d_{\mathrm{hyp}}(z,\eta w))^{-\nu} F\Bigl(\frac{\nu}{2},\frac{\nu}{2}+\frac{1}{2};\nu+\frac{1}{2};\frac{1}{\cosh (d_{\mathrm{hyp}}(z,\eta w))^{2}}\Bigr)
=\mathcal{K}_{\nu}(z,w),
\end{align*}
for $\mathrm{Re}(\nu)=1+\delta>1$. This completes the proof.
\end{proof}

\vskip .10in
\begin{corollary}\label{cor_greens_rep}
Let $\delta \in (0,1/2)$. Then, for $z, w\in M$ with $z\not=w$, and $s\in\mathbb{C}$ with $\Re(s)>1$,
we have the representation
\begin{align*}
G_s(z,w)=\int\limits_{\mathrm{Re}(\nu)=1+\delta}^{}
\frac{c(\nu)}{(s-\frac{1}{2})^{2}-(\frac{1}{2}-\nu)^{2}}\,
 \mathcal{K}_{\nu}(z,w)d\nu,
\end{align*}
where $b_k(\nu):=(\frac{\nu}{2})_{k} (\frac{\nu+1}{2})_{k}/(\nu+\frac{1}{2})_{k}$
and with $c(\nu):=2^{-\nu-1}\Gamma(\nu)/\left(\Gamma (\nu-\frac{1}{2})\pi^{3/2}\,i\right)$.
\end{corollary}

\begin{proof}
First, we note that $Q_{s-1}(1+2u(z,w))=Q_{s-1}(\cosh (d_{\mathrm{hyp}}(z,w)))$.
Hence, by formula 7.213 from \cite{GR07} with $a=s-1/2$ and $b=d_{\mathrm{hyp}}(z,w)$,
we have, for $\Re (s)>1/2$, the equality
\begin{align*}
Q_{s-1}(1+2u(z,w))
=\frac{1}{2}
\int\limits_{-\infty}^{\infty}r\tanh (\pi r)h(r)
P_{-\frac{1}{2}+ir}(\cosh (d_{\mathrm{hyp}}(z,w)))dr
\end{align*}
with
$$
h(r):=\frac{1}{(s-\frac{1}{2})^{2}+r^{2}}\,.
$$
The function $h(r)$ is obviously even
and holomorphic in the strip
$\left| \mathrm{Im}(r) \right|\leq \frac{1}{2} + \epsilon$, where $\epsilon>0$ is such that $\epsilon< \Re(s)-1/2$.
Therefore, $h(r)$ satisfies conditions (S1) and (S2) of subsection \ref{2.3}.
Moreover, we have the bound
$
h(r)=O\left( (1+\vert r\vert)^{-2}\right),
$
as $\mathrm{Im}(r)\rightarrow\pm\infty$ in the domain of condition (S2).
With all this, the conditions
of Proposition \ref{KernelviaCoshThm_new} are satisfied and the asserted representation
can be immediately derived.
\end{proof}

\vskip .10in
\begin{definition}\rm
For $\nu,\beta \in \mathbb{C}$ with $\mathrm{Re}(\beta)>\mathrm{Re}(\nu-1)>0$, and
$u\in\mathbb{R}^{+}$, we set
\begin{align}\label{def_Gsbeta_tilde}
\tilde{G}_{\nu,\beta}(u)&:=g_{\nu}(u)\tanh(u)^{\beta} \,
F\Bigl(\frac{\nu}{2},\frac{\nu}{2}+\frac{1}{2};\nu+\frac{1}{2};\frac{1}{\cosh(u)^2}\Bigr),
\end{align}
where $g_{\nu}(u)$ is given by \eqref{g_s}.
\end{definition}

\vskip .10in
\begin{theorem}\label{ParEisViaWaveTh}
Let $\nu,\beta \in \mathbb{C}$ with $\Re(\beta )> \Re (\nu+1)> 2$.
\begin{enumerate}
\item[(i)]
For $z, w\in M$, the wave distribution $\mathcal{W}_{M,\frac{1}{4}}(z,w)(\tilde{G}_{\nu,\beta })$ is well-defined
and holomorphic as a function of $\beta$.
\item[(ii)] For $z, w\in M$ with $z \neq w$ and $s\in\mathbb{C}$ with $\mathrm{Re}(s)>1$,
the wave distribution $\mathcal{W}_{M,\frac{1}{4}}(z,w)(\tilde{G}_{\nu,\beta })$ admits a holomorphic continuation
in $\beta$ to $\{\beta \in
\mathbb{C}\,|\,\Re(\beta)>0\} \cup \{0\}$, we have the representation
\begin{align*}
G_s(z,w)=
\int\limits_{\mathrm{Re}(\nu)=1+\delta}^{} \frac{c(\nu)}{(s-\frac{1}{2})^{2}-(\nu-\frac{1}{2})^{2}}\,
\mathcal{W}_{M,\frac{1}{4}}(z,w)(\tilde{G}_{\nu,\beta })\Big\vert _{\beta =0}d\nu,
\end{align*}
for some $\delta>0$ sufficiently small, and with $c(\nu)$ defined in Corollary \ref{cor_greens_rep}.
\end{enumerate}
\end{theorem}

\begin{proof}
The proof is analogous to the method of proof of Theorem \ref{EllEisViaWaveTh} in section 6.

\vskip .10in
\emph{Part (i):} We first employ the equality
$$
\tilde{G}_{\nu,\beta}(u)= \sum\limits_{k=0}^{\infty}
\frac{\left(\frac{\nu}{2}\right)_{k} \left(\frac{\nu}{2}+\frac{1}{2}\right)_{k}}{k!\,\left(\nu+\frac{1}{2}\right)_k }g_{\nu+2k, \beta}(u),
$$
where $g_{\nu,\beta}(u)$ is given by \eqref{def_gsbeta}. Along the lines of the proof of Lemma
\ref{intformula} and Lemma \ref{lemma_>8},
we then derive, for $\nu, \beta \in \mathbb{C}$ with $\Re(\nu)>1$ and $\Re (\beta)>\Re(\nu)+8$,
the relation
\begin{align*}
H(r,\tilde{G}_{\nu, \beta})=\sum\limits_{k=0}^{\infty}
\frac{\left(\frac{\nu}{2}\right)_{k} \left(\frac{\nu}{2}+\frac{1}{2}\right)_{k}}{k!\,\left(\nu+\frac{1}{2}\right)_k } H(r, g_{\nu+2k, \beta});
\end{align*}
here, for $\mathrm{Re} (\nu)=1+\delta$, we employed the bound
$$
\left|\frac{\left(\frac{\nu}{2}\right)_{k}\left(\frac{\nu}{2}+\frac{1}{2}\right)_{k}}
{k!\,\left(\nu+\frac{1}{2}\right)_k } \right| \ll \left|\frac{\left(\frac{\nu}{2}\right)_{k} }{k!} \right|,
$$
as $k \to \infty$. Finally, applying Lemma \ref{bound_with_beta} and following the lines of the proof of Theorem \ref{EllEisViaWaveTh},
we deduce the statement (i).

\vskip .10in
\emph{Part (ii):}
Again, following lines of the proof of Theorem \ref{EllEisViaWaveTh} (ii), we deduce that, for $z, w\in M$ with $z \neq w$,
the wave distribution $\mathcal{W}_{M,\frac{1}{4}}(z,w)(\tilde{G}_{\nu,\beta })$ admits a holomorphic continuation
in $\beta$ to $\{\beta \in
\mathbb{C}\,|\,\Re(\beta)>0\} \cup \{0\}$ such that
$$
\mathcal{W}_{M,\frac{1}{4}}(z,w)(\tilde{G}_{\nu,\beta })\Big\vert _{\beta =0} = \sum_{k=0}^{\infty} \frac{b_k(\nu)}{k!}K_{\nu + 2k}(z,w)= \mathcal{K}_{\nu}(z,w).
$$
Hence, the statement of Corollary \ref{cor_greens_rep} completes the proof of the theorem.
\end{proof}

Combining Theorem \ref{ParEisViaWaveTh} with Proposition \ref{prop_rel_par_greens}
yields the wave representation of the parabolic Eisenstein series.


\vspace{5mm}
\noindent

\noindent
Jay Jorgenson \\
Department of Mathematics \\
The City College of New York \\
Convent Avenue at 138th Street \\
New York, NY 10031
U.S.A. \\
e-mail: jjorgenson@mindspring.com

\vspace{5mm}
\noindent
Anna-Maria von Pippich \\
Fachbereich Mathematik \\
Technische Universit\"at Darmstadt \\
Schlo{\ss}gartenstr. 7 \\
D-64289 Darmstadt \\
Germany \\
e-mail: pippich@mathematik.tu-darmstadt.de

\vspace{5mm}\noindent
Lejla Smajlovi\'c \\
Department of Mathematics \\
University of Sarajevo\\
Zmaja od Bosne 35, 71 000 Sarajevo\\
Bosnia and Herzegovina\\
e-mail: lejlas@pmf.unsa.ba

\begin{thebibliography}{JKvP10}

\bibitem[Ab77]{Ab77} Abikoff, W.: \emph{Degenerating families of Riemann surfaces}.
Annals of Math. \textbf{105}, (1977), 29--44.


\bibitem[Be95]{Be95}
Beardon, A. F.: \emph{The geometry of discrete groups}.
Graduate Texts in Mathematics \textbf{91}, Springer-Verlag, New York, 1995.

\bibitem[Ch84]{Ch84} Chavel, I.: \emph{Eigenvalues in Riemannian geometry}.
Pure Appl. Math. \textbf{115}, Academic Press, Inc., Orlando, FL, 1984.

\bibitem[CdV81]{CdV81}
Colin de Verdi\`ere, Y.: \emph{Une nouvele d\'emonstration du prolongement m\'eromorphe des s\'eries d'Eisenstein}.
C. R. Acad. Sci. Paris S\'er. I Math. \textbf{293}, (1981), 361--363.


\bibitem[DG75]{DG75}
Duistermaat, J. J. and Guillemin, V. W.: \emph{The spectrum of positive elliptic operators and periodic bicharacteristics}.
Invent. Math. \textbf{29} (1975), 39--79.

\bibitem[Fa07]{Fa07}
Falliero, T.: \emph{D\'eg\'en\'erescence de s\'eries d'Eisenstein hyperboliques}.
Math. Ann. \textbf{339} (2007), 341--375.

\bibitem[GJM08]{GJM08}
Garbin, D., Jorgenson, J., and Munn, M.: \emph{On the appearance of Eisenstein series through degeneration}.
Comment. Math. Helv. \textbf{83} (2008), 701--721.

\bibitem[GvP09]{GvP09}
Garbin, D. and von Pippich, A.-M.: \emph{On the behavior of Eisenstein series through elliptic degeneration}.
Comm. Math. Phys. \textbf{292} (2009), 511--528.

\bibitem[GR07]{GR07} Gradshteyn, I. S. and Ryzhik, I. M.:  \emph{Table of integrals, series and products}.
Elsevier Academic Press, Amsterdam, 2007.


\bibitem[He83]{He83}
Hejhal, D.: \emph{The Selberg trace formula for ${\rm PSL}(2,\mathbb{R})$.} II,
Lecture Notes in Math. \textbf{1001}, Springer-Verlag, Berlin, 1983.

\bibitem[Iwa02]{Iwa02} Iwaniec, H.: \emph{Spectral methods of automorphic forms}.
Graduate Studies in Mathematics \textbf{53}, American Mathematical Society, Providence, RI, 2002.

\bibitem[JK11]{JK11} Jorgenson, J. and Kramer, J.:
\emph{Sup-norm bounds for automorphic forms and Eisenstein series}.
In:. J. Cogdell, J. Funke, M. Rapoport, and T. Yang (eds.), Arithmetic Geometry and Modular Forms. Advanced Lectures in
Mathematics 19, 407--444, International Press, Bejing, 2011.

\bibitem[JKvP10]{JKvP10} Jorgenson, J., Kramer, J., and von Pippich, A.-M.:
\emph{On the spectral expansion of hyperbolic Eisenstein series}.
Math. Ann. \textbf{346} (2010), 931--947.

\bibitem[JLa03]{JLa03} Jorgenson, J. and Lang, S.: \emph{Analytic continuation and identities
involving heat, Poisson, wave and Bessel kernels}.
Math. Nachr. \textbf{258} (2003), 44--70.

\bibitem[JvPS15]{JvPS15} Jorgenson, J., von Pippich, A.-M., and Smajlovi\'c, L.:
\emph{Applications of Kronecker's limit formula for elliptic Eisenstein series}.
Submitted for publication.

\bibitem[Ku73]{Ku73}
Kubota, T.: \emph{Elementary theory of Eisenstein series.} Kodansha
Ltd., Tokyo, 1973.

\bibitem[KM79]{KM79}
Kudla, S. S. and Millson, J. J.: \emph{Harmonic differentials and closed geodesics on a Riemann surface}.
Invent. Math. \textbf{54} (1979), 193--211.

\bibitem[PS85]{PS85}
Phillips, R. S. and Sarnak, P.: \emph{On cusp forms for co-finite subgroups of ${\rm PSL}(2,\mathbb{R})$.}
Invent. Math. \textbf{80} (1985), 339--364.

\bibitem[vP10]{vP10} von Pippich, A.-M.: \emph{The arithmetic of elliptic Eisenstein
series}. PhD thesis, Humboldt-Universit\"{a}t zu Berlin, 2010.

\bibitem[vP15]{vP15}
von Pippich, A.-M.: \emph{A Kronecker limit type formula for elliptic Eisenstein series},
in preparation.


\bibitem[Ri04]{Ri04}
Risager, M. S.:  \emph{On the distribution of modular symbols for compact surfaces.}
Int. Math. Res. Not. (2004), 2125--2146.

\end{thebibliography}
\end{document}